\documentclass[11pt]{amsart}

\usepackage{amsmath,amssymb,amsfonts,mathrsfs,extarrows}
\usepackage[cal=boondoxo, scr=boondoxo]{mathalfa}

\DeclareFontFamily{U}{wncy}{}
\DeclareFontShape{U}{wncy}{m}{n}{<->wncyr10}{}
\DeclareSymbolFont{mcy}{U}{wncy}{m}{n}
\DeclareMathSymbol{\Sh}{\mathord}{mcy}{"58}

\usepackage{xcolor}
\usepackage{etoolbox} % 用于自动绑定环境别名，解决交叉引用名称丢失问题
\usepackage{tikz}
\usepackage{tikz-cd}
\usepackage{subfiles} 
\usetikzlibrary{matrix, arrows.meta}

\usepackage[
    top=1in,          % 顶部边距
    bottom=0.85in,    % 底部边距
    left=1in,         % 左边距
    right=1in,        % 右边距
    headheight=15pt,  % 设置页眉本身的高度
    headsep=0.35in,   % 增大页眉与正文的间距
    footskip=0.45in   % 正文到页码的距离
]{geometry}

\usepackage{verbatim, url}
\usepackage[pdfpagelabels, bookmarksopen=true]{hyperref}

\usepackage[nameinlink, capitalize]{cleveref}

\numberwithin{equation}{subsection}

\theoremstyle{plain}
\newtheorem{theorem}[subsection]{Theorem}
\newtheorem{corollary}[subsection]{Corollary}
\newtheorem{proposition}[subsection]{Proposition}
\newtheorem{lemma}[subsection]{Lemma}

\newtheorem{situation}[subsection]{Situation}
\newtheorem{setup}[subsection]{Setup}

\theoremstyle{definition}
\newtheorem{definition}[subsection]{Definition}
\newtheorem{remark}[subsection]{Remark}

\newtheorem*{acknowledge}{Acknowledgement}

\newtheoremstyle{pgstyle}{}{}{}{}{}{\textbf{.}}{ }{\textbf{\thmname{#1}\thmnumber{#2}}\thmnote{ (#3)}}
\theoremstyle{pgstyle}
\newtheorem{pg}[subsection]{}

\AtBeginEnvironment{theorem}{\crefalias{subsection}{theorem}}
\AtBeginEnvironment{corollary}{\crefalias{subsection}{corollary}}
\AtBeginEnvironment{proposition}{\crefalias{subsection}{proposition}}
\AtBeginEnvironment{lemma}{\crefalias{subsection}{lemma}}
\AtBeginEnvironment{fact}{\crefalias{subsection}{fact}}
\AtBeginEnvironment{conjecture}{\crefalias{subsection}{conjecture}}
\AtBeginEnvironment{exercise}{\crefalias{subsection}{exercise}}
\AtBeginEnvironment{notation}{\crefalias{subsection}{notation}}
\AtBeginEnvironment{observation}{\crefalias{subsection}{observation}}
\AtBeginEnvironment{situation}{\crefalias{subsection}{situation}}
\AtBeginEnvironment{setup}{\crefalias{subsection}{setup}}
\AtBeginEnvironment{definition}{\crefalias{subsection}{definition}}
\AtBeginEnvironment{remark}{\crefalias{subsection}{remark}}
\AtBeginEnvironment{question}{\crefalias{subsection}{question}}
\AtBeginEnvironment{example}{\crefalias{subsection}{example}}
\AtBeginEnvironment{problem}{\crefalias{subsection}{problem}}
\AtBeginEnvironment{pg}{\crefalias{subsection}{pg}}

\DeclareMathOperator{\Br}{Br}
\DeclareMathOperator{\sm}{sm}
\DeclareMathOperator{\GL}{GL}
\DeclareMathOperator{\PGL}{PGL}
\DeclareMathOperator{\Aut}{Aut}

\DeclareMathOperator{\Hom}{Hom}
\DeclareMathOperator{\Pic}{Pic}

\DeclareMathOperator{\Spec}{Spec}

\DeclareMathOperator{\fppf}{fppf}
\DeclareMathOperator{\desc}{desc}
\DeclareMathOperator{\im}{im}
\DeclareMathOperator{\sep}{sep}

\newcommand{\G}{\mathbb{G}}
\newcommand{\Z}{\mathbb{Z}}

\renewcommand{\H}{\mathrm{H}}
\newcommand{\A}{\mathbf{A}}

\newcommand{\ms}[1]{\mathscr{#1}}
\newcommand{\mc}[1]{\mathcal{#1}}

\newcommand{\mr}[1]{\mathrm{#1}}

\newcommand{\et}{\operatorname{\acute{e}t}}

\let\oldsection\section
\renewcommand{\section}{\clearpage\oldsection}

\begin{document}
\title[Descent and Brauer--Manin Obstructions on Deligne--Mumford Stacks]{Descent and Brauer--Manin Obstructions on Deligne--Mumford Stacks}

\author[D. Li]{Donghao Li}
\address{Shanghai Center for Mathematical Sciences, Fudan University, 2005 Songhu Road, 200438 Shanghai, China}
\email{ridongen031110@gmail.com}

\author[S. Liu]{Sheng Liu}
\address{Shanghai Center for Mathematical Sciences, Fudan University, 2005 Songhu Road, 200438 Shanghai, China}
\email{Liushengisjr@outlook.com}

\author[C. Lv]{Chang Lv}
\address{State Key Laboratory of Cyberspace Security Defense\\
Institute of Information Engineering\\
Chinese Academy of Sciences\\
Beijing 100093, China}
\email{lvchang@amss.ac.cn}

\author[C. Zhang]{Chen Zhang}
\address{School of Mathematical Sciences, Capital Normal University, 105 Xisanhuanbeilu, 100048 Beijing, China}
\email{1230403014@cnu.edu.cn}

\begin{abstract} 
We generalize and compare local--global obstructions for algebraic stacks over number fields.
For smooth separated Deligne--Mumford stacks of finite type with a quasi-projective coarse
moduli space, we prove that the descent obstruction is contained in the Brauer--Manin obstruction. 
By lifting
$\G_m$-gerbes, we obtain an inclusion between the corresponding composite
obstructions. 
We also show that in this setting the descent obstruction coincides with both the \'{e}tale Brauer--Manin obstruction and the
iterated descent obstruction.
The Brauer--Manin obstruction also coincides with the iterated Brauer--Manin obstruction.

\end{abstract}
\date{\today}
\maketitle
\setcounter{tocdepth}{1}
\tableofcontents

\section{Introduction}
Let $k$ be a number field and let $X$ be a variety over $k$.
The study of rational points is often approached through the local-global principle, which compares $X(k)$ with the space $X(\A_k)$ of adelic points. To measure the failure of adelic points to arise from rational points, Manin \cite{Manin1971BrauerGrothendieck} introduced the Brauer-Manin obstruction via the Brauer-Grothendieck group.

Meanwhile, descent theory using torsors was developed by Colliot-Th\'{e}l\`{e}ne and Sansuc in \cite{ColliotTheleneSansuc1987} for commutative algebraic groups. Harari and Skorobogatov \cite{HarariSkorobogatov2002NonAbelian} subsequently generalized descent theory to algebraic groups and compared the descent obstruction with the Brauer-Manin obstruction. Combining finite \'{e}tale coverings with the Brauer-Manin obstruction gives rise to the \'{e}tale Brauer-Manin obstruction. A series of works \cite{Stoll2007} \cite{Demarche2009} \cite{Skorobogatov2009} \cite{Poonen2010} \cite{cao2018comparingdescentobstructionbrauermanin} established that the \'{e}tale Brauer-Manin obstruction and the descent obstruction coincide. Building on these results, Cao \cite{Cao_2020} further showed that the iterated descent obstruction is equivalent to the descent obstruction.

Beyond varieties, algebraic stacks arise naturally in moduli problems, and their arithmetic therefore calls for stack-theoretic analogues of these obstruction theories. Extending arithmetic obstruction theory from scheme-theoretic varieties to stack-theoretic spaces is non-trivial: points of a stack form groupoids before passing to isomorphism classes, the map from rational points to adelic points need not be injective, and torsors, gerbes, and cohomology must be treated on the big sites of the stack.

A general framework for points and cohomological obstructions on categories fibred in groupoids was developed in \cite{Lv21}, while the Brauer-Manin pairing and several descent results for algebraic stacks were established in \cite{LW23}.

The purpose of the present paper is to continue this program by establishing a unified comparison theory for descent, \'{e}tale Brauer-Manin, and composite obstructions on Deligne-Mumford stacks, thereby extending the fundamental results of Cao-Demarche-Xu \cite[1.5]{cao2018comparingdescentobstructionbrauermanin} and Cao \cite[1.2]{Cao_2020} to the stack-theoretic setting.

Our first main result is the following:
\begin{theorem}[\cref{desc_contained_in_Br}]
    Let $\mathcal{X}$ be a smooth separated Deligne-Mumford stack of finite type over $k$ with quasi-projective coarse moduli space. 
    Then:
    \begin{align*}
        \mathcal X(\A_k)^{\mathrm{desc}}
 \subseteq \mathcal X(\A_k)^{\mathrm{PGL}}
 = \mathcal X(\A_k)^{\mathrm{Br}},
    \end{align*}
    which generalizes the results in \cite{HarariSkorobogatov2002NonAbelian} and \cite{Poo17}.
\end{theorem} 

The equality on the right is obtained by relating Azumaya algebras to $\mathrm{PGL}_n$-torsors.  
More precisely, the boundary map associated with
\[
 1\longrightarrow \G_m\longrightarrow \mathrm{GL}_n
 \longrightarrow \mathrm{PGL}_n\longrightarrow 1
\]
identifies the obstruction defined by a constant-rank Azumaya algebra with the obstruction defined by its associated $\mathrm{PGL}_n$-torsor.  Surjectivity of the Brauer map in the Deligne-Mumford case, as established in \cite{SHIN-THESIS2019}, then allows all Brauer classes to be detected in this way.

We next compare the \'{e}tale Brauer-Manin obstruction with the descent obstruction.
Combining this with our previous result gives one inclusion.

Meanwhile
in \cite{LW26}
the authors used the $\mathrm{SL}_n$-descent technique to prove 
the reverse inclusion for certain quotient stacks.
By handling the geometrically integral condition as in \cite{Cao_2020}, we prove: 
\begin{theorem}[\cref{et_br_eq_desc_2}]

    Let $\mathcal{X}$ be a smooth separated Deligne-Mumford stack of finite type over $k$ with quasi-projective coarse moduli space. 
    Then:
\[
  \mathcal X(\A_k)^{\mathrm{\acute et},\mathrm{Br}}
  =\mathcal X(\A_k)^{\mathrm{desc}}
  =\mathcal X(\A_k)^{\mathrm{desc},\mathrm{desc}},
\]
which generalizes the results in \cite{Cao_2020} and \cite{cao2018comparingdescentobstructionbrauermanin}, removing the geometrical integrality condition in Cao's paper.
\end{theorem}

One expects that for nice varieties or algebraic stacks over $k$,
there are only two kinds of obstructions.
We therefore expect that the iterated Brauer-Manin (resp. second descent, connected, PGL, ...) obstruction is just the Brauer-Manin obstruction (which coincides with the second descent, connected, PGL obstruction by \cite{harari}).
Indeed, we prove the following:
\begin{theorem}[\cref{Br_Br_eq_Br}, \cref{thm:br_torsion}]
Let $\mc{X}$ be a quasi-separated regular Noetherian algebraic stack with generically affine stabilizer.
Then:
\begin{enumerate}
    \item $\Br(\ms{X})$ is torsion.
    \item If $\ms X$ is furthermore of finite type over $k$, then \(\mathcal X(\A_k)^{\mathrm{Br},\mathrm{Br}}
  =\mathcal X(\A_k)^{\mathrm{Br}}.\)
\end{enumerate}
The first statement generalizes \cite[2.1.8]{SHIN-THESIS2019}.
\end{theorem}
Other potential equalities remain open.

Under our stronger assumptions, \cite[Theorem 1.2]{LW26} yields the additional identities displayed below:
\[
\begin{tikzcd}
{\ms{X}(\A_k)^{\desc,\et,\Br}} \arrow[r, equal] & {\ms{X}(\A_k)^{\et,\Br}} \arrow[r, equal] \arrow[d, hook] & {\ms{X}(\A_k)^{\desc,\desc}} \arrow[r, equal]  & {\ms{X}(\A_k)^{\mathrm{fin},\desc}} \arrow[r, equal] & \ms{X}(\A_k)^{\desc} \arrow[d, hook] \\
& \ms{X}(\A_k)^{\Br} \arrow[r, equal]                       & \ms{X}(\A_k)^{2\text{-desc}} \arrow[rr, equal] &                                                                    & \ms{X}(\A_k)^{\mathrm{conn}}        
\end{tikzcd}
\]

The paper is organized as follows: In Section 2, we introduce the obstruction formalism for algebraic stacks that will be used throughout the paper, including descent, Brauer–Manin, PGL, and \'{e}tale obstructions, together with twisting and the corresponding composite obstructions. 
In Section 3, we compare the descent and Brauer–Manin obstructions for smooth separated Deligne–Mumford stacks with quasi-projective coarse moduli space. We prove that the PGL obstruction coincides with the Brauer–Manin obstruction and establish a functorial comparison between their composite versions using lifting $\mathbb{G}_m$-gerbes. 
In Section 4, we compare the \'{e}tale Brauer–Manin obstruction with the descent obstruction and prove that the \'{e}tale Brauer–Manin, descent, and iterated descent obstructions coincide under the same hypotheses. 
In Section 5, we compare the Brauer–Manin obstruction with its iterated version; the main ingredient is the surjectivity of the pullback on Brauer groups along $\mathbb{G}_m$-gerbes, from which we deduce that the Brauer–Manin obstruction is unchanged by iteration. 
Appendix A collects the required results on Azumaya algebras and Brauer groups of algebraic stacks, while Appendix B establishes a torsionness result for the Brauer group of a class of regular algebraic stacks with generically affine stabilizers.

\begin{acknowledge}
    The authors would like to thank Yang Cao and Xucheng Zhang for helpful discussions. 
Part of this work was carried out during visits by all authors to the School of Mathematical Sciences, Peking University, and by the third author to the Morningside Center of Mathematics, Chinese Academy of Sciences. The authors are grateful for the hospitality provided by these institutions.
\end{acknowledge}

\section{Obstructions on algebraic stacks}\label{section:2}

\begin{setup}
    Let $\mathrm{Stack}/k$ denote the (2,1)-category of algebraic stacks over $k$.
\end{setup}

\begin{definition}
Let $F : \ms{C} \to \ms{D}$ 
be a functor from a 2-category to a category 
forgetting 2-morphisms. 
We say that $F$ is \textcolor{blue}{stable} 
if for any 1-morphisms $f$ and $g$ of $\ms{C}$ that are 2-isomorphic,
we have $F(f) = F(g)$.
When $\ms{D}$ is viewed as a 2-category, 
this is equivalent to $F$ being a 2-functor.
Clearly if $F$ is stable, it remains so when restricted to a sub 2-category of $\ms{C}$.
\end{definition}

\begin{definition}
Let $F : (\mathrm{Stack}/k)^{\mathrm{op}} \to \mathrm{Sets}$ 
be a stable functor. 
Let  $\ms X$ and $T$ be algebraic stacks over $k$ and $A \in F(\ms{X})$. 
The \textcolor{blue}{evaluation} of a $T$-point $x \in \ms{X}(T) = \mathrm{Mor}_{\mathrm{Stack}/k}(T, \ms{X})$ at $A$ is defined to be the
image of $A$ under the pull-back map $F(x) : F(\ms{X}) \to F(T)$ induced by $x$, denoted by $A(x)$.
\end{definition}

Then we will have the following commutative diagram for every $A\in F(\ms{X})$:
\[
\begin{tikzcd}
\ms{X}(k) \arrow[r, "-\circ q"] \arrow[d, "A(-)"] & \ms{X}(\A_k) \arrow[d, "A(-)"]\\
F(k) \arrow[r, "F(q)"] & F(\A_k)
\end{tikzcd}
\]
where $q: \Spec\A_k\to \Spec k$ is induced by the natural inclusion $k\hookrightarrow \A_k$.

\begin{definition}
    The \textcolor{blue}{obstruction given by $A$} 
    is the full subcategory $\ms{X}(\A_k)^A$ of $\ms{X}(\A_k)$ whose objects
are characterized by:
\begin{align*}
        \ms{X}(\A_k)^A := \{x \in \ms{X}(\A_k) : A(x) \in \im F(q)\}.
\end{align*}
The \textcolor{blue}{$F$-obstruction} 
is the full subcategory $\ms{X}(\A_k)^F$ 
whose objects are characterized by:
\begin{align*}
        \ms{X}(\A_k)^F := \bigcap_{A \in F(\ms{X})} \ms{X}(\A_k)^A = \{x \in \ms{X}(\A_k) : \im F(x)\subseteq \im F(q)\}.
\end{align*}
\end{definition}

\begin{remark}
By abuse of notation, we identify $\ms{X}(k)$ with its image in $\ms{X}(\A_k)$ because in general the map $\ms{X}(k)\xrightarrow{-\circ q} \ms{X}(\A_k)$ is not injective.
    Hence we immediately obtain:
    \begin{align*}
        \ms{X}(k) \subseteq \ms{X}(\A_k)^F \subseteq \ms{X}(\A_k)^A\subseteq \ms{X}(\A_k).
    \end{align*}
    Since $F$ is stable, it follows that $\ms{X}(\A_k)^F$ and $\ms{X}(\A_k)^A$ are full subcategories of $\ms{X}(\A_k)$.
\end{remark}

\begin{remark}
    According to \cite[2.29, 3.14]{Lv21},
    for any linear $k$-group $G$,
    the functors $\check{\H}^1_{\fppf}(-,G)$, $\check{\H}^1_{\et}(-,G)$, ${\H}^1_{\fppf}(-,G)$, ${\H}^1_{\et}(-,G)$, ${\H}^2_{\et}(-,G)$ (when $G$ is commutative) are all stable.
    Hence we can define cohomological obstructions and descent obstructions for algebraic stacks.
\end{remark}

\begin{definition}
    The \textcolor{blue}{descent obstruction} is:
    \begin{align*}
        \ms{X}(\A_k)^{\desc} := \bigcap_{\text{linear $k$-group $G$}}\ms{X}(\A_k)^{\check{\H}^1_{\fppf}(-,G)}.
    \end{align*}
    The \textcolor{blue}{Brauer-Manin obstruction} is:
    \begin{align*}
        \ms{X}(\A_k)^{\Br} := \ms{X}(\A_k)^{{\H}^2_{\et}(-,\G_m)}.
    \end{align*}
    The \textcolor{blue}{PGL obstruction} is:
    \begin{align*}
        \ms{X}(\A_k)^{\PGL} := \bigcap_{n\in\mathbb{N}^*}\ms{X}(\A_k)^{\check{\H}^1_{\fppf}(-,\PGL_n)}.
    \end{align*}
    The \textcolor{blue}{\'etale obstruction} is:
    \begin{align*}
        \ms{X}(\A_k)^{\et} := \bigcap_{\text{finite \'etale linear $k$-group } G}\ms{X}(\A_k)^{\check{\H}^1_{\fppf}(-,G)}.
    \end{align*}
\end{definition}

\begin{pg}
    Let $G$ be a linear $k$-group.
    Let $\mathrm{Tors}(\ms{X}_{\fppf},G)$ (resp. $\mathrm{Tors}(\ms{X}_{\et},G)$) 
    denote the groupoid of all $G$-torsors over $\ms{X}_{\fppf}$ (resp. $\ms{X}_{\et}$). 
    The isomorphism classes of torsors make $\mathrm{Tors}(\ms{X}_{\fppf},G)_{/\cong}$
    (resp. $\mathrm{Tors}(\ms{X}_{\et},G)_{/\cong}$)
    a pointed set.

    First, we claim that for any linear  $k$-group $G$ (automatically smooth by \cite[\href{https://stacks.math.columbia.edu/tag/0BF6}{Tag 0BF6}]{SP})
    \begin{align*}
    \check{\H}_{\et}^{1}(\ms{X},G)\simeq \mathrm{Tors}(\ms{X}_{\et},G)_{/\cong}\simeq\mathrm{Tors}(\ms{X}_{\fppf},G)_{/\cong}\simeq \check{\H}_{\fppf}^{1}(\ms{X},G)
    \end{align*} 
    The first and third isomorphisms come from \cite[(2.9)]{LW23}.
    For the second isomorphism, it suffices to show that for any $G$-torsor $\ms{P}\to \ms{X}$ in the big fppf topology, 
    it is also a $G$-torsor in the big \'etale topology.

    To see this, let $U \to \ms{X}$ be an arbitrary object in the big site $\ms{X}_{\et}$ (where $U$ is a scheme). We need to show that the pullback $\ms{P}_U := \ms{P} \times_{\ms{X}} U \to U$ can be trivialized by an \'etale covering of $U$.
    
    Since $\ms{P}$ is an fppf $G$-torsor over $\ms{X}$, its pullback $\ms{P}_U$ is an fppf $G$-torsor over the scheme $U$. 
    Because $G$ is an affine group scheme, the morphism $\ms{P}_U \to U$ is affine, which implies that $\ms{P}_U$ is a scheme.
    Furthermore, since $G$ is smooth over $k$, the trivial torsor $U \times G \to U$ is a smooth surjective morphism. 
    By the fppf descent of smoothness, the morphism of schemes $\ms{P}_U \to U$ is also smooth and surjective.

    Now, we can apply \cite[\href{https://stacks.math.columbia.edu/tag/055U}{Tag 055U}]{SP} to the smooth surjective morphism of schemes $\ms{P}_U \to U$. There exists an \'etale covering $\{V_i \to U\}$ such that $\ms{P}_U \to U$ admits a section over each $V_i$.
    This implies that $\ms{P}_U$ is trivialized over the \'etale covering $\{V_i \to U\}$. 
    Since this holds for any scheme $U$ over $\ms{X}$, $\ms{P}$ is indeed a $G$-torsor in the big \'etale topology.

    Hence, for a linear $k$-group $G$,
    there is no difference between big fppf $G$-torsors and big \'etale $G$-torsors.
    Henceforth we will not specify the topology of $G$-torsors in question.
\end{pg}

\begin{corollary}
    Let $\ms{X}$ be an algebraic stack over $k$ and $G$ be a linear $k$-group. Then:
    \begin{align*}
        \ms{X}(\A_k)^{\check{\H}_{\fppf}^{1}(-,G)}=\ms{X}(\A_k)^{\check{\H}_{\et}^{1}(-,G)}.
    \end{align*}
\end{corollary}
\begin{proof}
    This follows immediately from the above discussion and the following diagram:
    \begin{center}\begin{tikzpicture}[>=stealth] 
\matrix[matrix of math nodes,row sep=2em, column sep=2em, text height=1.5ex, text depth=0.25ex] 
{ 
|[name=21]| \check{\H}_{\et}^{1}(\ms{X},G) & |[name=22]| \check{\H}_{\fppf}^{1}(\ms{X},G) \\
|[name=31]| \check{\H}_{\et}^{1}(\A_k,G) & |[name=32]| \check{\H}_{\fppf}^{1}(\A_k,G) \\
|[name=41]| \check{\H}_{\et}^{1}(k,G) & |[name=42]| \check{\H}_{\fppf}^{1}(k,G) \\
}; 
\draw[->,font=\scriptsize]
(21) edge (31) 
(41) edge (31) 
(22) edge (32) 
(42) edge (32) 
(21) edge node[above=-1pt] {$\simeq$} (22) 
(31) edge node[above=-1pt] {$\simeq$} (32) 
(41) edge node[above=-1pt] {$\simeq$} (42); 
\end{tikzpicture} 
\end{center} 
\end{proof}

\begin{definition}
Let $G$ be a linear $k$-group and let $\ms{X}$ be an algebraic stack over
$k$. Let
$
    \sigma\in\check{\mathrm{H}}^1_{\fppf}(k,G).
$ 
Choose a right $G$-torsor $\ms{P}_{\sigma}$ over $k$ representing $\sigma$.
Set
\[
    G^{\sigma}:=\operatorname{ad}(\ms{P}_{\sigma})
    =\underline{\operatorname{Aut}}_{G}(\ms{P}_{\sigma}),
\]
where $\underline{\operatorname{Aut}}_{G}(\ms{P}_{\sigma})$ denotes the
sheaf of $G$-equivariant automorphisms of $\ms{P}_{\sigma}$. Then
$\ms{P}_{\sigma}$ is naturally a $(G^{\sigma},G)$-bitorsor. We denote its
inverse $(G,G^{\sigma})$-bitorsor by $\ms{P}_{\sigma}^{\circ}$.

Let $f:\ms{Y}\xrightarrow{G}\ms{X}$
be a right $G$-torsor over $\ms{X}$, 
and 
$
    \ms{P}_{\sigma,\ms{X}}:=\ms{P}_{\sigma}\times_k \ms{X}.
$
The \textcolor{blue}{twist} of $\ms{Y}$ by $\sigma$ is the contracted product
\[
    f^{\sigma}:\ms{Y}^{\sigma}
    :=
    \ms{Y}\mathbin{\times^{G}}\ms{P}_{\sigma,\ms{X}}^{\circ}
    \longrightarrow \ms{X}.
\]
It is naturally a right $G^{\sigma}$-torsor over $\ms{X}$.

Equivalently, there is a natural isomorphism of
$G^{\sigma}$-torsors
\[
    \ms{Y}^{\sigma}
    \simeq
    \underline{\operatorname{Isom}}_{G}
    (\ms{P}_{\sigma,\ms{X}},\ms{Y}),
\]
where $G^{\sigma}$ acts by
precomposition.

Up to isomorphism, $G^{\sigma}$ and $f^{\sigma}$ depend
only on the class $\sigma$ and not on the chosen representative
$\ms{P}_{\sigma}$.
\end{definition}

\begin{definition}
    Let $G$ be a commutative linear $k$-group and $\ms{X}$ be an algebraic stack over $k$.
    Let $\tau\in \H^2_{\et}(k,G)$ and let $G^\tau$ be a $G$-gerbe over $k_{\et}$ representing $\tau$.
    For any $G$-gerbe $g:\ms{Z}\to\ms{X}$  over $\ms{X}$ in the big \'etale topology, 
    we define the \textcolor{blue}{twist} of $\ms{Z}$ by $\tau$ to be the algebraic stack:
    \begin{align*}
        g^\tau:\ms{Z}^\tau=\ms{Z}\times^G G^\tau=[\ms{Z}\times G^\tau/G] \to\ms{X}
    \end{align*}
\end{definition}

\begin{remark}
We use the definition of gerbes in \cite{OLSSON},
which under our assumptions are algebraic stacks.
    For more general definitions of twist of torsors and gerbes,
    cf. \cite[3.30, 4.20]{Lv21}.
    From now on, when mentioning gerbes and torsors, we only consider the big \'etale topology.
\end{remark}

\begin{proposition}
    \cite[3.20]{Lv21}.
    Let $\ms{X}$ be an algebraic stack over $k$
    and $f\in \check{\H}^1_{\fppf}(\ms{X},G)$
    be represented by a $G$-torsor $f:\ms{Y}\xrightarrow{G}\ms{X}$ over $\ms{X}$.
    Then we have:
    \begin{align*}
        \ms{X}(\A_k)^{f}=\bigcup_{\sigma\in \check{\H}^1_{\fppf}(k,G)}f^\sigma(\ms{Y}^\sigma(\A_k)).
    \end{align*}
\end{proposition}

\begin{proposition}
    \cite[4.6]{Lv21}.
    Let $\ms{X}$ be an algebraic stack over $k$
    and $g\in {\H}^2_{\et}(\ms{X},G)$
    be represented by a $G$-gerbe $g:\ms{Z}\to\ms{X}$ over $\ms{X}$.
    Then we have:
    \begin{align*}
        \ms{X}(\A_k)^{g}=\bigcup_{\tau\in \H^2_{\et}(k,G)}g^\tau(\ms{Z}^\tau(\A_k)).
    \end{align*}
\end{proposition}

\begin{definition}
    Let $\circledcirc $ be any obstruction appearing in this section.
    We can define \textcolor{blue}{composite obstructions}:
    for any linear $k$-group $G$:
    \begin{align*}
        \ms{X}(\A_k)^{\check{\H}_{\fppf}^1(-,G),\circledcirc} := \bigcap_{f\in \check{\H}^1_{\fppf}(\ms{X},G)}\bigcup_{\sigma\in \check{\H}^1_{\fppf}(k,G)}f^\sigma(\ms{Y}^\sigma(\A_k)^{\circledcirc});
    \end{align*}
    and for any commutative linear $k$-group $G$:
    \begin{align*}
        \ms{X}(\A_k)^{{\H}_{\et}^2(-,G),\circledcirc} := \bigcap_{g\in {\H}_{\et}^2(\ms{X},G)}\bigcup_{\tau\in {\H}_{\et}^2(k,G)}g^\tau(\ms{Z}^\tau(\A_k)^{\circledcirc}),
    \end{align*}
    where $f\in \check{\H}^1_{\fppf}(\ms{X},G)$ ranges over all $G$-torsors over $\ms{X}$ and similarly $g\in {\H}_{\et}^2(\ms{X},G)$ ranges over all $G$-gerbes over $\ms{X}$.

    We now define several composite obstructions that will be used later.
    The \textcolor{blue}{\'etale-Brauer obstruction} is:
    \begin{align*}
        \ms{X}(\A_k)^{\et,\Br}:=\bigcap_{G \text{ finite \'etale $k$-group}}\bigcap_{f\in \check{\H}^1_{\fppf}(\ms{X},G)}\bigcup_{\sigma\in \check{\H}^1_{\fppf}(k,G)}f^\sigma(\ms{Y}^\sigma(\A_k)^{\Br}).
    \end{align*}
    The \textcolor{blue}{iterated descent obstruction} is:
    \begin{align*}
        \ms{X}(\A_k)^{\desc,\desc}:=\bigcap_{G \text{ linear $k$-group}}\bigcap_{f\in \check{\H}^1_{\fppf}(\ms{X},G)}\bigcup_{\sigma\in \check{\H}^1_{\fppf}(k,G)}f^\sigma(\ms{Y}^\sigma(\A_k)^{\desc}).
    \end{align*}
    The \textcolor{blue}{iterated Brauer obstruction} is:
    \begin{align*}
        \ms{X}(\A_k)^{\Br,\Br}:=\bigcap_{g\in {\H}^2_{\et}(\ms{X},\G_m)}\bigcup_{\tau\in {\H}^2_{\et}(k,\G_m)}g^\tau(\ms{Z}^\tau(\A_k)^{\Br}).
    \end{align*}
    Other composite obstructions such as $\ms{X}(\A_k)^{\desc,\Br}, \ms{X}(\A_k)^{\Br,\desc} \ldots$ are defined similarly.
\end{definition}

\section{Comparing descent obstruction and Brauer-Manin obstruction}\label{section:3}

\begin{definition}[Brauer map]
    By \cite[1.4.2]{SHIN-THESIS2019} and \cref{sm_top_et_top_coincide}, the assignment $\mc{A} \mapsto [\ms{G}_{\mc{A}}]$ induces an injective group homomorphism:
    \begin{align*}
        \alpha_{\ms{X}}' : \mathrm{Br}_{\mathrm{Az}}(\ms{X}) \to {\H}^2_{\et}(\ms{X}, \G_m).
    \end{align*}
    By \cref{every_class_has_constant_rank} and \cref{azumaya_algebra_torsion},
    if $\ms{X}$ is of finite type over $k$,
    then the homomorphism factors through the torsion subgroup ${\H}^2_{\et}(\ms{X}, \G_m)_{\mathrm{tors}}$, giving the restriction:
    \begin{align*}
        \alpha_{\ms{X}} : \mathrm{Br}_{\mathrm{Az}}(\ms{X}) \to {\H}^2_{\et}(\ms{X}, \G_m)_{\mathrm{tors}}.
    \end{align*}
    We call $\mathrm{Br}'(\ms{X}):={\H}^2_{\et}(\ms{X}, \G_m)_{\mathrm{tors}}$ the \textcolor{blue}{cohomological Brauer group} of $\ms{X}$, 
    and $\mathrm{Br}(\ms{X}):={\H}^2_{\et}(\ms{X}, \G_m)$ the \textcolor{blue}{Brauer-Grothendieck group} of $\ms{X}$, which we also call the Brauer group.
\end{definition}

\begin{theorem}\label{desc_contained_in_Br}
   Let $\ms{X}$ be a smooth separated Deligne-Mumford stack of finite type over $k$ with quasi-projective coarse moduli space. Then:
   \begin{align*}
   \ms{X}(\A_k)^{\desc} \subseteq \ms{X}(\A_k)^{\PGL}=\ms{X}(\A_k)^{\Br}
   \end{align*}
\end{theorem}
We prove the theorem below.

\begin{lemma}\label{lem:surj_brauer_map}\cite[2.1.5]{SHIN-THESIS2019}.
Let $\ms{X}$ be a smooth separated Deligne-Mumford stack with quasi-projective coarse moduli space. 
The Brauer map $\alpha_{\ms{X}}$ is surjective.
\end{lemma}

\begin{remark}
    The original statement contains the ``generically tame'' condition.
    But we only consider a number field $k$.
    Since \(\operatorname{char} k=0\),
    this condition is automatically satisfied.
    We will henceforth omit the ``generically tame'' condition.
\end{remark}

\begin{remark}
    Indeed, the coarse moduli space is automatically 
    a scheme, by \cite[Page 25]{Knutson1971}.
\end{remark}

\begin{lemma}\label{lem:br_prime_equals_br}\cite[2.1.8]{SHIN-THESIS2019}. 
If $\ms{X}$ is a regular Noetherian Deligne-Mumford stack 
(in particular if $\ms{X}$ satisfies the hypotheses of \cref{desc_contained_in_Br}), 
then $\mathrm{Br}(\ms{X})=\mathrm{Br}'(\ms{X})$.
\end{lemma}

\subsection{Proof of \cref{desc_contained_in_Br}}
The following natural map is injective (see \cite[8.1.6]{Poo17}):
\begin{align*}
    \check{\H}^1_{\et}(\A_k,\PGL_n)\hookrightarrow \prod_{v}\check{\H}^1_{\et}(k_v,\PGL_n)
\end{align*}
and there is a natural isomorphism (see \cite[8.1.6]{Poo17}):
\begin{align*}
    \Br(\A_k)\simeq \bigoplus_{v}\Br(k_v).
\end{align*}

For any element $\gamma \in \Br(\ms{X})$,
by 
\cref{every_class_has_constant_rank},
\cref{lem:surj_brauer_map},
\cref{lem:br_prime_equals_br},
we can find an Azumaya algebra $\mc{A}$ on $\ms{X}$ of constant rank $n^2$ such that $\alpha'_{\ms{X}}([\mc{A}])=\gamma$.
Let $\beta=\beta(\gamma)\in \check{\H}^1_{\et}(\ms{X}, \PGL_n)$ be the element induced by $\mc{A}$.
By \Cref{lem:gerbe_class_coincides}, the homomorphism $\check{H}^1_{\et}(\ms{X}, \PGL_n) \to \Br(\ms{X})[n]$
maps $\beta$ to $\gamma$.

As in the proof of \cite[8.5.2]{Poo17},
consider the following commutative diagram:
\begin{center}\begin{tikzpicture}[>=stealth] 
\matrix[matrix of math nodes,row sep=2em, column sep=2em, text height=1.5ex, text depth=0.25ex] 
{ 
|[name=21]| \check{\H}_{\et}^{1}(\ms{X},\PGL_n) & |[name=22]| \Br(\ms{X})[n] \\
|[name=31]| \prod_{v}\check{\H}^1_{\et}(k_v,\PGL_n) & |[name=32]| \prod_{v}\Br(k_v)[n] \\
|[name=41]| \check{\H}_{\et}^{1}(k,\PGL_n) & |[name=42]| \Br(k)[n] \\
}; 
\draw[->,font=\scriptsize]
(21) edge (31) 
(41) edge (31) 
(22) edge (32) 
(42) edge (32) 
(21) edge node[above=-1pt] {} (22) 
(31) edge node[above=-1pt] {$\simeq$} (32) 
(41) edge node[above=-1pt] {$\simeq$} (42); 
\end{tikzpicture} 
\end{center} 
If $\check{H}^1_{\et}(\ms{X}, \PGL_n) \to \Br(\ms{X})[n]$ maps $a$ to $A$,
then $\ms{X}(\A_k)^{a}=\ms{X}(\A_k)^{A}.$
    
Thus:
\begin{align*}
    \ms{X}(\A_k)^{\Br}&=\bigcap_{\gamma\in \Br(\ms{X})}\ms{X}(\A_k)^{\gamma}=\bigcap_{\beta(\gamma)}\ms{X}(\A_k)^{\beta}
    \supseteq \bigcap_{n\in\mathbb{N}^*}\ms{X}(\A_k)^{\check{\H}_{\et}^{1}(-,\PGL_n)}\\
    &=\bigcap_{n\in\mathbb{N}^*}\bigcap_{\gamma\in \im (\check{H}^1_{\et}(\ms{X}, \PGL_n) \to \Br(\ms{X})[n])} \ms{X}(\A_k)^{\gamma} \supseteq \bigcap_{n\in\mathbb{N}^*}\ms{X}(\A_k)^{\Br[n]} = \ms{X}(\A_k)^{\Br},
\end{align*}
all the above inclusions are in fact equalities.
Hence we have:
\begin{align*}
    \ms{X}(\A_k)^{\desc}&\subseteq \ms{X}(\A_k)^{\PGL}
    =\bigcap_{n\in\mathbb{N}^*}\ms{X}(\A_k)^{\check{\H}_{\fppf}^{1}(-,\PGL_n)}
    =\bigcap_{n\in\mathbb{N}^*}\ms{X}(\A_k)^{\check{\H}_{\et}^{1}(-,\PGL_n)}
    =\ms{X}(\A_k)^{\Br}.
\end{align*}

To pass from the comparison
\(
\ms X(\mathbf A_k)^{\mathrm{PGL}}
=
\ms X(\mathbf A_k)^{\mathrm{Br}}
\)
to a comparison of the corresponding composite obstructions, it is not enough to compare the cohomology classes of a $\mathrm{PGL}_n$-torsor and its image in $\H^2_{\et}(\ms X,\mathbb G_m)$. We need a compatible comparison of their twists. The next lemma provides exactly this: it realizes the boundary class by the lifting gerbe and constructs morphisms from each twisted torsor to the corresponding twisted gerbe.

\begin{lemma}\label{lem:torsor_to_gerbe_coincides}
Let $\ms{X}$ be a smooth separated Deligne-Mumford stack of finite type over $k$ with quasi-projective coarse moduli space. 
    For any $n\in\mathbb{N}^*$ and any $\PGL_n$-torsor $f:\ms{Y}\to \ms{X}$ over $\ms{X}$, 
    we can construct a $\G_m$-gerbe $g:\ms{G}_{\ms{Y}}\to \ms{X}$ over $\ms{X}$ and a morphism of stacks $\pi:\ms{Y}\to \ms{G}_{\ms{Y}}$ such that:
    \begin{enumerate}
        \item The following diagram is commutative:
        \[
        \begin{tikzcd}
            \ms{Y} \arrow[r, "\pi"] \arrow[dr, "f"'] & \ms{G}_{\ms{Y}} \arrow[d, "g"] \\
            & \ms{X}    
        \end{tikzcd}
        \]
        \item The class $[g]$ is the image of $[f]$ along the connecting map $\delta_{\ms{X}}: \check{\H}^1_{\et}(\ms{X}, \PGL_n) \to {\H}^2_{\et}(\ms{X}, \G_m)$.
        \item Let $\sigma\in\check{\mathrm{H}}^1_{\mathrm{\acute et}}(k,\mathrm{PGL}_n)$
    and choose a right $\mathrm{PGL}_n$-torsor $\ms{P}_{\sigma}$ over $k$
representing $\sigma$. Put
\[
    \tau:=\delta_k(\sigma)
    \in\mathrm{H}^2_{\mathrm{\acute et}}(k,\mathbb{G}_m).
\]
Let 
$
    \mathrm{PGL}_n^{\sigma}
    :=\operatorname{ad}(\ms{P}_{\sigma})=\underline{\mathrm{Aut}}_{\PGL_n}(\ms{P}_{\sigma})
$
be the corresponding inner form of $\mathrm{PGL}_n$. Twisting the
standard central extension by $\ms{P}_{\sigma}$ gives a central extension
\[
    1\longrightarrow\mathbb{G}_m
    \longrightarrow\mathrm{GL}_n^{\sigma}
    \xrightarrow{\,b^{\sigma}\,}
    \mathrm{PGL}_n^{\sigma}
    \longrightarrow1.
\]
The twist 
$f^{\sigma}:\ms{Y}^{\sigma}\longrightarrow \ms{X}$ 
is a $\mathrm{PGL}_n^{\sigma}$-torsor. Let
$
    g_{\ms{Y}^{\sigma}}^{\sigma}:
    \mathcal{G}_{\ms{Y}^{\sigma}}^{\sigma}\longrightarrow \ms{X}
$
be its lifting gerbe with respect to the above twisted central
extension. Then there is a natural equivalence of
$\mathbb{G}_m$-gerbes over $\ms{X}$: 
\[
    \mathcal{G}_{\ms{Y}^{\sigma}}^{\sigma}
    \simeq
    \mathcal{G}_{\ms{Y}}^{-\tau}.
\]
In particular, there is a morphism
$
    \pi^{\sigma}:\ms{Y}^{\sigma}
    \longrightarrow\mathcal{G}_{\ms{Y}}^{-\tau}
$
such that the following diagram is commutative:
\[
\begin{tikzcd}
\ms{Y}^{\sigma}
    \arrow[rr,"\pi^{\sigma}"]
    \arrow[dr,"f^{\sigma}"']
&&
\mathcal{G}_{\ms{Y}}^{-\tau}
    \arrow[dl,"g_{\ms{Y}}^{-\tau}"]
\\
& \ms{X}.&
\end{tikzcd}
\]
    \end{enumerate}
\end{lemma}

\iffalse
\begin{proof}
    The construction is the same as in \cite[12.2.5, 12.2.6]{OLSSON}.
    The verification is direct.
\end{proof}
\fi

\begin{proof}
\cite[12.2.5, 12.2.6]{OLSSON}.
Let
\[
1\longrightarrow\G_m
\longrightarrow \GL_n
\overset{b}{\longrightarrow}
\PGL_n
\longrightarrow 1
\]
be the standard exact sequence of sheaves of groups on the big étale site of
$\mathcal X$. We regard the $\PGL_n$-torsor
$f:\mathcal Y\to \mathcal X$ as a $\PGL_n$-torsor $\mathcal P$ on this site.

We define a fibered category
$
g:\mathscr G_{\mathcal Y}\longrightarrow \mathcal X
$
as follows: for an object $T\to \mathcal X$, the fiber
$\mathscr G_{\mathcal Y}(T)$ is the groupoid whose objects are pairs
$
(\widetilde{\mathcal P},\epsilon),
$ 
where $\widetilde{\mathcal P}$ is a $\GL_n$-torsor on the big étale site of
$T$, and
\[
\epsilon:b_*\widetilde{\mathcal P}\xrightarrow{\sim}\mathcal P|_T
\]
is an isomorphism of $\PGL_n$-torsors. Here
$b_*\widetilde{\mathcal P}$ denotes the torsor obtained from
$\widetilde{\mathcal P}$ by extension of structure group along
$b:\GL_n\to \PGL_n$.

A morphism
$
(\widetilde{\mathcal P}',\epsilon')
\longrightarrow
(\widetilde{\mathcal P},\epsilon)
$
over a morphism $T'\to T$ is an isomorphism of $\GL_n$-torsors
$
\widetilde{\mathcal P}|_{T'}\xrightarrow{\sim}\widetilde{\mathcal P}'
$
compatible (after applying $b_*$) with the two given isomorphisms to
$\mathcal P|_{T'}$. This is precisely the lifting gerbe associated to the
$\PGL_n$-torsor $\mathcal P$ and the above exact sequence.

\iffalse
Since $\mathcal P$ is locally trivial for the étale topology, locally on
$\mathcal X$ it is isomorphic to the trivial $\PGL_n$-torsor. Over such a
trivializing étale cover, the trivial $\GL_n$-torsor gives an object of
$\mathscr G_{\mathcal Y}$. Thus $\mathscr G_{\mathcal Y}$ is locally nonempty.
Moreover, for any object
\[
(\widetilde{\mathcal P},\epsilon)\in \mathscr G_{\mathcal Y}(T),
\]
its automorphism sheaf consists of automorphisms of the $\GL_n$-torsor
$\widetilde{\mathcal P}$ which induce the identity on
$b_*\widetilde{\mathcal P}$. After étale localization this automorphism sheaf
is the kernel of 
$
\GL_n\longrightarrow \PGL_n,
$ namely $\G_m$. Since $\G_m$ is central in $\GL_n$, these local
identifications glue canonically. Hence
$
g:\mathscr G_{\mathcal Y}\longrightarrow \mathcal X
$
is a $\G_m$-gerbe.

\fi

We now construct the morphism
$
\pi:\mathcal Y\longrightarrow \mathscr G_{\mathcal Y}.
$
Let $(T\to \mathcal X) \in \mc X_{\et}$ 
and let
$
y\in \mathcal Y(T)
$
be a section of the $\PGL_n$-torsor $\mathcal P|_T$.
Then
\[
\epsilon_y:\PGL_{n,T}\xrightarrow{\sim}\mathcal P|_T,
\qquad
h\longmapsto y\cdot h .
\]
Let $\GL_{n,T}$ be the trivial $\GL_n$-torsor on $T$. Since
$
b_*\GL_{n,T}\simeq \PGL_{n,T},
$
we define
$
\pi(y):=(\GL_{n,T},\epsilon_y)\in \mathscr G_{\mathcal Y}(T).
$

This construction is compatible with pullback in $T$, and therefore defines a
morphism of stacks
$
\pi:\mathcal Y\longrightarrow \mathscr G_{\mathcal Y}
$ over $\mc X$. 

It remains to identify the class of the gerbe. Choose an étale cover 
$\{U_i\to \mathcal X\}$ trivializing the $\PGL_n$-torsor $\mathcal P$, and
choose local sections
$
y_i\in \mathcal P(U_i).
$

On $U_{ij}=U_i\times_{\mathcal X}U_j$ there are unique elements
$
a_{ij}\in \PGL_n(U_{ij})
$
such that
\[
y_j=y_i\cdot a_{ij}.
\]
The elements $a_{ij}$ form a Čech $1$-cocycle representing the class
\[
[f]\in \check \H^1_{\et}(\mathcal X,\PGL_n).
\]
After refining the cover if necessary, choose lifts
$
\widetilde a_{ij}\in \GL_n(U_{ij})
$
of the $a_{ij}$. Since
$
a_{ij}a_{jk}=a_{ik}
$
on $U_{ijk}$, the element
$
\widetilde a_{ij}\widetilde a_{jk}\widetilde a_{ik}^{-1}
$
lies in the kernel of $\GL_n\to \PGL_n$, hence defines an element
\[
c_{ijk}\in \G_m(U_{ijk}).
\]
The collection $(c_{ijk})$ is a Čech $2$-cocycle with values in
$\G_m$, and by the definition of the connecting map associated to
\[
1\to \G_m\to \GL_n\to \PGL_n\to 1,
\]
it represents the image of $[f]$ in
$
\check \H^2_{\et}(\mathcal X,\G_m).
$

On the other hand, the local sections $y_i$ give local objects
\[
\xi_i:=\pi(y_i)=(\GL_{n,U_i},\epsilon_{y_i})
\]
of the gerbe $\mathscr G_{\mathcal Y}$ over $U_i$. On $U_{ij}$, the lift
$\widetilde a_{ij}$ gives an isomorphism
$
\xi_j\xrightarrow{\sim}\xi_i
$
of objects of $\mathscr G_{\mathcal Y}$. Indeed, multiplication by
$\widetilde a_{ij}$ on the trivial $\GL_n$-torsor induces multiplication by
$a_{ij}$ on the associated $\PGL_n$-torsor, which is exactly the transition
from the trivialization defined by $y_j$ to the one defined by $y_i$.

On a triple overlap $U_{ijk}$, the two composites
\[
\xi_k\longrightarrow \xi_j\longrightarrow \xi_i
\qquad\text{and}\qquad
\xi_k\longrightarrow \xi_i
\]
differ by the automorphism of $\xi_i$ given by
\[
\widetilde a_{ij}\widetilde a_{jk}\widetilde a_{ik}^{-1}
=
c_{ijk}\in \G_m(U_{ijk}).
\]
Therefore the Čech $2$-cocycle defining the banded
$\G_m$-gerbe $\mathscr G_{\mathcal Y}$ is precisely
$(c_{ijk})$. Hence
the image of 
$\delta([f])\in
\check \H^2_{\et}(\mathcal X,\G_m)$
by $\check \H^2_{\et}(\mathcal X,\G_m)\to \H^2_{\et}(\mathcal X,\G_m)$
is just $[\ms{G}_{\ms{Y}}]$
,
which proves (2).

\vspace{1em}

For (3),
let
$
    a:\ms{X}\longrightarrow \Spec k
$ be the structure morphism. Let $\ms{P}_{\sigma}$ be a right
$\mathrm{PGL}_n$-torsor over $k$ representing $\sigma$. Since
$\mathrm{PGL}_n$ acts on $\mathrm{GL}_n$ and $\mathrm{PGL}_n$ by conjugation, we may define
\[
    \mathrm{GL}_n^{\sigma}
    :=
    \ms{P}_{\sigma}
    \mathbin{\times^{\mathrm{PGL}_n}}
    \mathrm{GL}_n
\]
\[
    \mathrm{PGL}_n^{\sigma}
    :=
    \ms{P}_{\sigma}
    \mathbin{\times^{\mathrm{PGL}_n}}
    \mathrm{PGL}_n
    \simeq
    \operatorname{ad}(\ms{P}_{\sigma}).
\]
Since the conjugation action of $\mathrm{PGL}_n$ on the
center $\G_m\subseteq\mathrm{GL}_n$ is trivial, we have a central extension
\[
    1\longrightarrow\G_m
    \longrightarrow\mathrm{GL}_n^{\sigma}
    \xrightarrow{\,b^{\sigma}\,}
    \mathrm{PGL}_n^{\sigma}
    \longrightarrow1.
\]

Let
$
    \mathcal{G}_{\ms{P}_{\sigma}}
    \longrightarrow\Spec k
$
be the lifting gerbe of $\ms{P}_{\sigma}$ with respect to
\[
    1\longrightarrow\G_m
    \longrightarrow\mathrm{GL}_n
    \xrightarrow{\,b\,}
    \mathrm{PGL}_n
    \longrightarrow1.
\]
By (2), its class is
\[
    [\mathcal{G}_{\ms{P}_{\sigma}}]
    =
    \delta_k(\sigma)
    =:
    \tau
    \in\mathrm{H}^2_{\mathrm{\acute et}}(k,\mathbb{G}_m).
\]
Consequently, the inverse gerbe
$\mathcal{G}_{\ms{P}_{\sigma}}^{\circ}$ represents $-\tau$. By the
definition of the twist of a gerbe, we may therefore take
\[
    \mathcal{G}_{\ms{Y}}^{-\tau}
    =
    \mathcal{G}_{\ms{Y}}
    \mathbin{\times^{\G_m}}
    a^*\mathcal{G}_{\ms{P}_{\sigma}}^{\circ}.
\]

We construct an equivalence
$
    \Phi:
    \mathcal{G}_{\ms{Y}}^{-\tau}
    \longrightarrow
    \mathcal{G}_{\ms{Y}^{\sigma}}^{\sigma}.
$
Let $T\to \ms{X}$ be an object of the big \'{e}tale site of $\ms{X}$.
An object of $\mathcal{G}_{\ms{Y}}^{-\tau}(T)$ is
represented by a pair
\[
    (\widetilde{\mathscr{P}},\epsilon),
    \qquad
    (\widetilde{\mathscr{Q}},\eta)
\]
where $\widetilde{\mathscr{P}}$ and $\widetilde{\mathscr{Q}}$ are
$\mathrm{GL}_n$-torsors over $T$, and
$
    \epsilon:
    b_*\widetilde{\mathscr{P}}
    \xrightarrow{\sim}
    \ms{Y}|_T, \ 
    \eta:
    b_*\widetilde{\mathscr{Q}}
    \xrightarrow{\sim}
    \ms{P}_{\sigma,T},
$
with
$    
\ms{P}_{\sigma,T}:=\ms{P}_{\sigma}\times_kT.
$

Define
$
    \mathcal{I}(\widetilde{\mathscr{P}},\widetilde{\mathscr{Q}})
    :=
    \underline{\operatorname{Isom}}_{\mathrm{GL}_n}
    (\widetilde{\mathscr{Q}},\widetilde{\mathscr{P}}),
$
which is naturally a right torsor
under
$\underline{\operatorname{Aut}}_{\mathrm{GL}_n}(\widetilde{\mathscr{Q}}),
$
where an automorphism of $\widetilde{\mathscr{Q}}$ acts by precomposition.
There are natural isomorphisms
\[
\begin{aligned}
    \underline{\operatorname{Aut}}_{\mathrm{GL}_n}
    (\widetilde{\mathscr{Q}})
    &\simeq
    \widetilde{\mathscr{Q}}
    \mathbin{\times^{\mathrm{GL}_n}}
    \mathrm{GL}_n
    \simeq
    b_*\widetilde{\mathscr{Q}}
    \mathbin{\times^{\mathrm{PGL}_n}}
    \mathrm{GL}_n
\\
    &\xrightarrow[\eta]{\sim}
    \ms{P}_{\sigma,T}
    \mathbin{\times^{\mathrm{PGL}_n}}
    \mathrm{GL}_n
    =
    \mathrm{GL}_{n,T}^{\sigma}.
\end{aligned}
\]
Here the second isomorphism follows from the fact that the center
$\G_m$ acts trivially on $\mathrm{GL}_n$ by conjugation.
Thus
$
    \mathcal{I}(\widetilde{\mathscr{P}},\widetilde{\mathscr{Q}})
$
is a $\mathrm{GL}_{n}^{\sigma}$-torsor.

Moreover, extension of the structure group along
$b^{\sigma}$ gives
\[
\begin{aligned}
    (b^{\sigma})_*
    \mathcal{I}(\widetilde{\mathscr{P}},\widetilde{\mathscr{Q}})
    &\simeq
    \underline{\operatorname{Isom}}_{\mathrm{PGL}_n}
    (b_*\widetilde{\mathscr{Q}},b_*\widetilde{\mathscr{P}})
    \xrightarrow[\epsilon,\eta]{\sim}
    \underline{\operatorname{Isom}}_{\mathrm{PGL}_n}
    (\ms{P}_{\sigma,T},\ms{Y}|_T).
\end{aligned}
\]
By the definition of the twist by the inverse bitorsor, there is a
natural isomorphism
\[
    \underline{\operatorname{Isom}}_{\mathrm{PGL}_n}
    (\ms{P}_{\sigma,T},\ms{Y}|_T)
    \simeq
    \ms{Y}^{\sigma}|_T.
\]
Hence
\[
    \left(
    \mathcal{I}(\widetilde{\mathscr{P}},\widetilde{\mathscr{Q}}),
    (b^{\sigma})_*
    \mathcal{I}(\widetilde{\mathscr{P}},\widetilde{\mathscr{Q}})
    \xrightarrow{\sim}
    \ms{Y}^{\sigma}|_T
    \right)
\]
is an object of
$\mathcal{G}_{\ms{Y}^{\sigma}}^{\sigma}(T)$.

We now describe this construction on morphisms. Let
\[
    \alpha:\widetilde{\mathscr{P}}\longrightarrow\widetilde{\mathscr{P'}},
    \qquad
    \beta:\widetilde{\mathscr{Q}}\longrightarrow\widetilde{\mathscr{Q'}}
\]
be morphisms compatible with the given identifications of the
associated $\mathrm{PGL}_n$-torsors. We define
\[
    \underline{\operatorname{Isom}}_{\mathrm{GL}_n}
    (\widetilde{\mathscr{Q}},\widetilde{\mathscr{P}})
    \longrightarrow
    \underline{\operatorname{Isom}}_{\mathrm{GL}_n}
    (\widetilde{\mathscr{Q'}},\widetilde{\mathscr{P'}})
\]
by
\[
    \varphi\longmapsto
    \alpha\circ\varphi\circ\beta^{-1},
\]
which is compatible with pull-back in $T$.

It remains to verify that the construction descends to the
contracted product of gerbes. Let
$\lambda\in\G_m(T)$. If $\lambda$ acts simultaneously on
$\widetilde{\mathscr{P}}$ and $\widetilde{\mathscr{Q}}$ by scalar automorphisms, then the
induced automorphism of
$\underline{\operatorname{Isom}}_{\mathrm{GL}_n}
(\widetilde{\mathscr{Q}},\widetilde{\mathscr{P}})$ is
\[
    \varphi\longmapsto
    \lambda\circ\varphi\circ\lambda^{-1}
    =
    \varphi.
\]
Thus the construction is balanced for the
$\G_m$-actions and descends to a morphism of
$\G_m$-gerbes
\[
    \Phi:
    \mathcal{G}_{\ms{Y}}
    \mathbin{\times^{\G_m}}
    a^*\mathcal{G}_{\ms{P}_{\sigma}}^{\circ}
    \longrightarrow
    \mathcal{G}_{\ms{Y}^{\sigma}}^{\sigma}.
\]
Notice that the inverse gerbe is essential here: scalar
automorphisms $(\lambda,\mu)$ of
$(\widetilde{\mathscr{P}},\widetilde{\mathscr{Q}})$ act on the resulting torsor by
$
    \varphi\longmapsto
    \lambda\circ\varphi\circ\mu^{-1}.
$
Consequently, the induced homomorphism on the bands is
\[
    (\lambda,\mu)\longmapsto\lambda\mu^{-1}.
\]

We finally prove that $\Phi$ is an equivalence. This may be checked
locally on $T$. After passing to an \'{e}tale covering of $T$, both
$\ms{Y}|_T$ and $\ms{P}_{\sigma,T}$ are trivial
$\mathrm{PGL}_n$-torsors. We may then take the trivial
$\mathrm{GL}_n$-torsors as their liftings. Under these
trivializations, $\mathrm{GL}_n^{\sigma}$ and
$\mathrm{PGL}_n^{\sigma}$ become respectively $\mathrm{GL}_n$ and
$\mathrm{PGL}_n$, and $\Phi$ maps the standard object of the source
gerbe to the standard object of the target gerbe. It also induces
the identity on their bands $\G_m$. Hence $\Phi$ is locally
an equivalence of $\G_m$-gerbes, and therefore it is an
equivalence globally.

We have proved
\[
    \mathcal{G}_{\ms{Y}^{\sigma}}^{\sigma}
    \simeq
    \mathcal{G}_{\ms{Y}}
    \mathbin{\times^{\G_m}}
    a^*\mathcal{G}_{\ms{P}_{\sigma}}^{\circ}
    =
    \mathcal{G}_{\ms{Y}}^{-\tau}.
\]

Applying the construction in (1) to the
$\mathrm{PGL}_n^{\sigma}$-torsor
$\ms{Y}^{\sigma}\to \ms{X}$ and the twisted central extension gives a
canonical morphism
$
    \ms{Y}^{\sigma}
    \longrightarrow
    \mathcal{G}_{\ms{Y}^{\sigma}}^{\sigma}.
$
Composing it with the above equivalence gives
\[
    \pi^{\sigma}:
    \ms{Y}^{\sigma}
    \longrightarrow
    \mathcal{G}_{\ms{Y}}^{-\tau}.
\]
Both morphisms are over $\ms{X}$, and hence $g_{\ms{Y}}^{-\tau}\circ\pi^{\sigma}=f^{\sigma}.$
This proves (3).

\end{proof}

\begin{theorem}\label{PGL_F_contained_in_Br_F}
    Let $\ms{X}$ be a smooth separated Deligne-Mumford stack of finite type over $k$ with quasi-projective coarse moduli space.
    Let $F: (\mathrm{Stack}/k)^{\mathrm{op}}\to \mathrm{Sets}$ be a stable functor.
    Then we have:
    \begin{align*}
        \ms{X}(\A_k)^{\PGL,F} \subseteq \ms{X}(\A_k)^{\Br,F}.
    \end{align*}
\end{theorem}

\begin{proof}
    Let $x\in \ms{X}(\A_k)^{\PGL,F}$.
    Then by definition,
    for any $n\in\mathbb{N}^*$ and any $\PGL_n$-torsor $f:\ms{Y}\to \ms{X}$ over $\ms{X}$ in the big \'etale topology, 
    there exists $\sigma\in \check{\H}^1_{\et}(k, \PGL_n)$ and $y\in \ms{Y}^{\sigma}(\A_k)^F$ such that:
    \begin{align*}
        x= f^\sigma\circ y: \A_k\to \ms{X}
    \end{align*}

    By definition,
    $
        \im F(y)\subseteq \im F(q)
    $
    where $q: \Spec \A_k\to \Spec k$.
    By \cref{lem:torsor_to_gerbe_coincides},
    there exists a morphism $j: \ms{Y}\to \ms{G}_{\ms{Y}}$, 
    where $g:\ms{G}_{\ms{Y}}\to \ms{X}$ is a $\G_m$-gerbe, 
    and there exists a twisted morphism $j^\sigma: \ms{Y}^{\sigma}\to \ms{G}^{\sigma}_{\ms{Y}^{\sigma}}\simeq (\ms{G}_{\ms{Y}})^{-\tau}$,
    where $\tau$ is the image of $\sigma$ under the connecting map $\delta_{k}: \check{\H}^1_{\et}(k, \PGL_n) \to {\H}^2_{\et}(k, \G_m)$.
    Consider $z=j^{\sigma}\circ y: \A_k\to \ms{G}_{\ms{Y}}^{-\tau}$.
    Then we have:
    \begin{align*}
        \im F(z)=\im (F(y)\circ F(j^{\sigma}))\subseteq \im F(y)\subseteq \im F(q).
    \end{align*}
    Hence by $x=f^\sigma\circ y=g^{-\tau}\circ z$, we have:
    \begin{align*}
        x\in \bigcup_{\tau\in {\H}^2_{\et}(k, \G_m)}g^{-\tau}((\ms{G}_{\ms{Y}})^{-\tau}(\A_k)^F)
    \end{align*}
    
    Finally, by \cref{lem:surj_brauer_map} and \cref{lem:br_prime_equals_br},
    as $\ms{Y}\to \ms{X}$ varies,
    $\ms{G}_{\ms{Y}}$ can represent all the classes in $\Br(\ms{X})$.
    Hence we have:
    \begin{align*}
        \ms{X}(\A_k)^{\PGL,F} \subseteq \ms{X}(\A_k)^{\Br,F}.
    \end{align*} 
\end{proof}

\begin{remark}\label{rem_after_PGL_F_contained_in_Br_F}
    We can define the functoriality of obstructions as in \cite[Definition 5.2]{Lv21}.
    By a similar argument,
    we can further prove that for any functorial obstruction $\mathrm{obs}$,
    \begin{align*}
        \ms{X}(\A_k)^{\desc,\mathrm{obs}}\subseteq \ms{X}(\A_k)^{\PGL,\mathrm{obs}}\subseteq \ms{X}(\A_k)^{\Br,\mathrm{obs}}\subseteq \ms{X}(\A_k)^{\Br}\cap \ms{X}(\A_k)^{\mathrm{obs}}.
    \end{align*}
    
    In fact, all the obstructions introduced above and their compositions are functorial obstructions.
\end{remark}

\section{Comparing \'etale-Brauer obstruction and descent obstruction}

We now compare the \'etale–Brauer obstruction with the descent obstruction.
To prove the first inclusion, we use the description of the descent set in terms of twists of a finite torsor. 
\cref{lem:desc_eq_cup} reduces the problem to comparing the descent and Brauer–Manin sets on each twist, 
and this comparison is provided by \cref{desc_contained_in_Br}, 
once we know that the twists still satisfy its geometric hypotheses.
We therefore first verify that the geometric hypotheses of \cref{desc_contained_in_Br} are preserved under twisting.

\begin{theorem}\label{et_br_eq_desc}
    Let $\ms{X}$ be a smooth separated Deligne-Mumford stack of finite type over $k$ with quasi-projective coarse moduli space.
    Then we have:
    \begin{align*}
        \ms{X}(\A_k)^{\et,\Br}\supseteq\ms{X}(\A_k)^{\desc}
    \end{align*}
\end{theorem}

\iffalse
\begin{lemma}\label{lem:torsor_of_quotient_is_quotient}
    Let $\ms{X}=[X/G_0]$ for somifflae quasi-projective $k$-scheme $X$ and linear $k$-group $G_0$.
    Then for any linear $k$-group $G$ and a $G$-torsor $\ms{Y}\to\ms{X}$,
    there exists a quasi-projective $k$-scheme $Y$ such that:
    \begin{align*}
        \ms{Y}\simeq [Y/G_0].
    \end{align*}
\end{lemma}

\begin{proof}
    Consider:
    \begin{align*}
        Y = X\times_{\ms{X}} \ms{Y}.
    \end{align*}
    Then $\ms{Y}\simeq [Y/G_0]$.
    
    Base change of a $G$-torsor is still a $G$-torsor,
    hence 
    \begin{align*}
        Y\to X
    \end{align*}
    is a $G$-torsor.
    It is well-known that by fppf descent of affineness,
    $Y\to X$ is affine and $Y$ is a scheme.

    By $G$ is of finite type over $k$,
    consider locally and we have $Y\to X$ is of finite type.
    By \cite[0B3H]{SP},
    $Y\to X$ is quasi-projective,
    hence $Y$ is a quasi-projective $k$-scheme.
\end{proof}
\fi

\begin{lemma}\label{lem:desc_eq_cup}
    \cite[4.3]{WL24}.
    Let $\ms{X}$ be an algebraic stack over $k$
    and $G$ be a finite $k$-group.
    Then for any $G$-torsor $f:\ms{Y}\to\ms{X}$,
    \begin{align*}
        \ms{X}(\A_k)^{\desc}=\bigcup_{\sigma\in\check \H^1_{\et}(k,G)} f^{\sigma}(\ms{Y}^{\sigma}(\A_k)^{\desc})
    \end{align*}
\end{lemma}

\begin{lemma}\label{lem:torsor_keeps_good_condition}
    Let $\ms{X}$ be a smooth separated Deligne-Mumford stack of finite type over $k$ with quasi-projective coarse moduli space.
    Then for any linear $k$-group $G$,
    any $G$-torsor $f:\ms{Y}\to\ms{X}$ 
    and any $\sigma\in \check \H^1_{\fppf}(k,G)$,
    $\ms{Y}^{\sigma}$ is still a smooth separated Deligne-Mumford stack of finite type over $k$ with quasi-projective coarse moduli space.
\end{lemma}

\begin{proof}
    We only prove the case where $\sigma$ is trivial.
    Indeed, $f^\sigma:\mathcal Y^\sigma\to\mathcal X$ is a torsor under the inner form $G^\sigma$, 
    which is again a linear $k$-group, 
    so the same argument applies to every twist.

    Since $G$ is a linear $k$-group, it is affine and of finite type over $k$;
    since $k$ has characteristic 0,
    $G$ is smooth.
    Thus, as a \(G\)-torsor,
    $f:\ms{Y}\to \ms{X}$ is representable, smooth and surjective. 
    It follows immediately that $\ms{Y}$ is a smooth separated Deligne-Mumford stack of finite type over $k$.

It remains to show that the coarse moduli space $Y$ of $\ms{Y}$ is quasi-projective over $k$.

Since $\ms{X}$ and $\ms{Y}$ are Deligne-Mumford stacks of finite type over $k$ with finite inertia, by the Keel-Mori Theorem, there exist coarse moduli spaces $\pi_{\ms{X}}: \ms{X} \to X$ and $\pi_{\ms{Y}}: \ms{Y} \to Y$, where $X$ and $Y$ are algebraic spaces of finite type over $k$. By assumption, $X$ is a quasi-projective scheme over $k$.

By the universal property of coarse moduli spaces, the morphism $f: \ms{Y} \to \ms{X}$ induces a unique morphism of algebraic spaces $\bar{f}: Y \to X$ making the following diagram commute:
    \[
    \begin{tikzcd}
    \ms{Y} \arrow[r, "f"] \arrow[d, "\pi_{\ms{Y}}"' ] & \ms{X} \arrow[d, "\pi_{\ms{X}}"] \\
    Y \arrow[r, "\bar{f}"'] & X
    \end{tikzcd}
    \]
By \cite[3.2(d)]{abramovich},
we can find an \'etale covering $U\to X$,
a linearly reductive group scheme $H\to U$
acting on a finite morphism $V\to U$
such that:
\begin{align*}
    \ms{X}\times_{X} U\simeq [V/H].
\end{align*}
Furthermore, we may assume that $U$ and $V$ are affine:
if not,
write $U=\bigcup_{i\in I} U_i$ as an affine covering.
$X$ is of finite type over $k$, hence quasi-compact.
So we can find finitely many $U_j\in \{U_i\}_{i\in I}$ such that the following composition is \'etale surjective:
\begin{align*}
    \bigcup_{j} U_j \hookrightarrow U \to X.
\end{align*}
Replace \(U\) by \(U'=\bigsqcup_j U_j\), \(H\) by \(H'=H\times_U U'\), and \(V\) by \(V'=V\times_U U'\).
By finiteness of $V\to U$, $V'$ is also an affine scheme.

Consider the base change of $\ms{Y}\to \ms{X}$:
\begin{align*}
    \ms{Y}\times_{X} U\to \ms{X}\times_{X} U\xrightarrow{\simeq} [V/H],
\end{align*}
and then base change further:
\begin{align*}
    W:=(\ms{Y}\times_{X} U)\times_{[V/H]} V\to V.
\end{align*}
The morphism is affine since $\ms{Y}\to \ms{X}$ is affine.
Hence $W$ is an affine scheme.
Denote $V=\Spec A$ and $W=\Spec B$.
$W\to \ms{Y}\times_{X} U$ is an $H$-torsor, in particular smooth surjective.
Hence by \cite[\href{https://stacks.math.columbia.edu/tag/04T3}{Tag 04T3}]{SP},
\begin{align*}
    \ms{Y}\times_{X} U\simeq [W/(W\times_{\ms{Y}\times_{X} U} W)]\simeq [W/(W\times_U H)]=[W/H]
\end{align*}

By \cite[3.3]{abramovich},
the coarse moduli spaces of $\ms{Y}\times_X U$ and $\ms{X}\times_X U$ are, 
respectively, 
$Y\times_X U$ and $U$.
And by the proof in \cite[3.3]{abramovich},
the morphism $Y\times_X U\to U$ is actually:
\begin{align*}
    \Spec B^H\to\Spec A^H,
\end{align*}
which is affine.
Finally by \cite[\href{https://stacks.math.columbia.edu/tag/03WG}{Tag 03WG}]{SP},
$Y\to X$ is affine.

Moreover, since $Y$ is of finite type over
$k$ and $X$ is separated over $k$,
the graph factorization
\begin{align*}
    Y\xrightarrow{\text{closed immersion}} Y\times_k X\xrightarrow{\text{pr}} X
\end{align*}
is of finite type.

By \cite[Page 25]{Knutson1971},
$X$ is a quasi-projective $k$-scheme 
and then $Y\to X$ is an affine finite type morphism of schemes.
So by quasi-projectivity of $X$ and \cite[\href{https://stacks.math.columbia.edu/tag/0B3H}{Tag 0B3H}]{SP},
$Y$ is also a quasi-projective $k$-scheme.
\end{proof}

\subsection{Proof of \cref{et_br_eq_desc}}
First we prove:
\begin{align*}
    \ms{X}(\A_k)^{\desc}\subseteq \ms{X}(\A_k)^{\et,\Br}.
\end{align*}

By \cref{lem:torsor_keeps_good_condition},
for any finite \'etale linear $k$-group $G$,
any $G$-torsor $f:\ms{Y}\to\ms{X}$,
and any $\sigma\in\check \H^1_{\et}(k, G)$,
$\ms{Y}^{\sigma}$ is a smooth separated Deligne-Mumford stack over $k$ with quasi-projective coarse moduli space.
Then by \cref{desc_contained_in_Br},
\begin{align*}
    \ms{Y}^{\sigma}(\A_k)^{\desc}\subseteq \ms{Y}^{\sigma}(\A_k)^{\Br}.
\end{align*}

So by \cref{lem:torsor_keeps_good_condition} and \cref{lem:desc_eq_cup}:
\begin{align*}
    \ms{X}(\A_k)^{\et,\Br}
    &=\bigcap_{G\text{ finite \'etale linear $k$-group }}\bigcap_{f:\ms{Y}\xrightarrow{G}\ms{X}}\bigcup_{\sigma\in\check \H^1_{\et}(k,G)} f^{\sigma}(\ms{Y}^{\sigma}(\A_k)^{\Br})\\
    &\supseteq \bigcap_{G\text{ finite \'etale linear $k$-group }}\bigcap_{f:\ms{Y}\xrightarrow{G}\ms{X}}\bigcup_{\sigma\in\check \H^1_{\et}(k,G)} f^{\sigma}(\ms{Y}^{\sigma}(\A_k)^{\desc})\\
    &=\ms{X}(\A_k)^{\desc}.
\end{align*}
\qed

\cref{et_br_eq_desc} gives one direction of the desired comparison. 
For the reverse direction, 
we use \cref{et_Br_contained_in_desc_desc} below, 
which establishes the required inclusion for quotient stacks $[X/G]$ when $X$ is smooth, quasi-projective, and geometrically integral. A general Deligne–Mumford stack in our class admits a quotient presentation, 
but the presenting variety need not be geometrically integral. 
The main point of the remainder of this section is therefore to reduce an adelic point to a geometrically integral component to which \cref{et_Br_contained_in_desc_desc} can be applied.

\begin{theorem}\label{et_br_eq_desc_2}
    Let $\ms{X}$ be a smooth separated Deligne-Mumford stack of finite type over $k$ with quasi-projective coarse moduli space.
    Then we have:
    \begin{align*}
        \ms{X}(\A_k)^{\et,\Br}=\ms{X}(\A_k)^{\desc,\desc}=\ms{X}(\A_k)^{\desc}
    \end{align*}
\end{theorem}

We begin by choosing a quotient presentation with connected structure group. 
The connectedness of the group will be useful twice: it allows geometric integrality of the presenting variety to be detected from that of the quotient stack, and later ensures that the group action preserves the geometric connected components of the presenting variety.

\begin{lemma}\label{lem:modification_of_Kre08}
    (Modification of \cite[4.4]{Kre08}.)
    Let $\ms{X}$ be a smooth separated Deligne-Mumford stack of finite type over $k$ with quasi-projective coarse moduli space.
    Then $\ms{X}\simeq [X/G]$ for some smooth quasi-projective $k$-variety $X$ and connected linear $k$-group $G$. 

    Furthermore,
    $X$ can be chosen to be geometrically integral if $\ms{X}$ is also geometrically integral.
\end{lemma}

\begin{remark}\label{kre08}
    Here we give the original statement in \cite{Kre08}:
    Let $\ms{X}$ be a smooth separated generically tame Deligne-Mumford stack of finite type over $k$
    ($k$ is an arbitrary field) with quasi-projective coarse moduli space.
    Then $\ms{X}\simeq [X/G]$ for some quasi-projective $k$-scheme $X$ and linear $k$-group $G$.
\end{remark}

\subsection{Proof of \cref{lem:modification_of_Kre08}}
By \cref{kre08},
we obtain a quasi-projective $k$-scheme $X$ such that 
$\ms{X}\simeq [X/G]$.
Choose an embedding \(G\hookrightarrow GL_n\) for some \(n\), and consider
\begin{align*}
    Y := [X\times \GL_n/G].
\end{align*}
Then we have
$
    \ms{X} \simeq [Y/\GL_n]. 
$

By \cite[4.3]{Kre08},
since $\ms{X}$ has a quasi-projective coarse moduli space,
$Y$ is a quasi-projective $k$-scheme.
By smoothness of $\GL_n$,
$Y\to \ms{X}\simeq [Y/\GL_n]$ is smooth.
Hence by smoothness of $\ms{X}$,
$Y$ is a smooth quasi-projective $k$-variety.

We now prove $Y$ is geometrically integral by contradiction.
Suppose $Y_{\overline{k}}$ is not integral.
$Y_{\overline{k}}$ is a smooth quasi-projective $\overline{k}$-variety,
so by \cite[13.2.L]{Vakil2025},
$Y_{\overline{k}}$ is not connected.
Write
\begin{align*}
    Y_{\overline{k}}=Y_1\sqcup Y_2.
\end{align*}
$f:Y_{\overline{k}}\to [Y/\GL_n]_{\overline{k}}\simeq\ms{X}_{\overline{k}}$ is smooth, in particular, an open morphism.
So $f(Y_1),f(Y_2)$ are open in $\ms{X}_{\overline{k}}$.
By connectedness of $\GL_n$,
every fibre of $f$ is connected,
hence:
\begin{align*}
    \ms{X}_{\overline{k}}=f(Y_1)\sqcup f(Y_2).
\end{align*}
Indeed, if \(f(Y_1)\cap f(Y_2)\) were nonempty, the corresponding fibre would meet both open-and-closed subsets \(Y_1\) and \(Y_2\), contradicting its connectedness.
But this contradicts the connectedness of $\ms{X}_{\overline{k}}$.
\qed

\begin{lemma}\label{et_Br_contained_in_desc_desc}
\cite[Lemma 5.3]{LW26}
    Let $\ms{X}=[X/G]$ where $X$ is a smooth quasi-projective geometrically integral $k$-variety
and $G$ is a linear $k$-group.
Then:
\begin{align*}
    \ms{X}(\A_k)^{\et,\Br}\subseteq \ms{X}(\A_k)^{\desc,\desc}
\end{align*}
\end{lemma}

The quotient presentation supplied by \cref{lem:modification_of_Kre08} need not have a geometrically integral numerator. 
To isolate the relevant component, 
we record the scheme of geometric connected components of a smooth variety. 
A $k$-rational point of this finite étale scheme selects a component that is itself defined over $k$, and smoothness then makes this component geometrically integral.

\begin{lemma}\label{lem:existence_of_C}
Let $Y$ be a smooth $k$-variety. 
Then there exists a finite \'etale $k$-scheme
$
C=\pi_0(Y)
$ 
and a $k$-morphism
$
q_Y:Y\longrightarrow C
$ 
such that, after base change to $\bar k$, 
the fibres of
$q_Y$ are precisely the connected components of
$Y_{\bar k}$. 

If $c\in C(k)$, then the fibre
\[
Y_c:=Y\times_C\Spec k
\]
is smooth and geometrically integral. 

Moreover,
if $Y$ is
quasi-projective, then so is $Y_c$.
\end{lemma}

\begin{proof}
Since $Y$ is of finite type over $k$, 
the set $\pi_0(Y_{\bar k})$ is finite and carries a continuous action of $\Gamma_k=\mathrm{Gal}(\bar k/k)$
by
\cite[\href{https://stacks.math.columbia.edu/tag/038E}{Tag 038E}]{SP}.
By the equivalence between \'etale $k$-schemes and
continuous $\Gamma_k$-sets \cite[\href{https://stacks.math.columbia.edu/tag/03QR}{Tag 03QR}]{SP},
the finite $\Gamma_k$-set $\pi_0(Y_{\bar k})$ corresponds
to a finite \'etale $k$-scheme $C$.

Over $\bar k$, define
\[
q_{\bar k}:Y_{\bar k}\longrightarrow
C_{\bar k}
=
\coprod_{\Gamma\in\pi_0(Y_{\bar k})}\Spec\bar k
\]
by sending each connected component $\Gamma$ to the
corresponding point of $C_{\bar k}$. 
This morphism is
$\Gamma_k$-equivariant and hence descends to a
$k$-morphism 
$
q_Y:Y\longrightarrow C
$ by \cite[\href{https://stacks.math.columbia.edu/tag/040L}{Tag 040L}]{SP}.

Let $c\in C(k)$. By construction, $(Y_c)_{\bar k}$ is a
connected component of $Y_{\bar k}$.
It is smooth 
because it is an open subscheme of $Y_{\bar k}$; hence it
is integral by \cite[13.2.L]{Vakil2025}. 
Thus $Y_c$ is
geometrically integral.

$Y_c$ is smooth and is quasi-projective whenever
$Y$ is quasi-projective by base change \cite[\href{https://stacks.math.columbia.edu/tag/01VV}{Tag 01VV}]{SP}.
\end{proof}

To produce such a $k$-rational component from an adelic point, 
it is convenient to disregard the complex places. 
This causes no loss for the obstruction sets considered here: over $\mathbb C$, 
torsors under linear algebraic groups are trivial and the Brauer group vanishes. Consequently, the complex components of an adelic point impose no additional condition for the Brauer–Manin, descent, étale–Brauer, or iterated descent obstructions.

Let \(\Omega_k^{\mathbb C}\) 
denote the set of complex places of \(k\), 
and let
\[
\mathbf A_k^{\mathrm{nc}}:=\prod_{v\notin\Omega_k^{\mathbb C}}' k_v
\]
be the ring of adeles away from the complex places. 
Thus
\[
\mathbf A_k\simeq
\mathbf A_k^{\mathrm{nc}}\times
\prod_{v\in\Omega_k^{\mathbb C}}k_v.
\]
As in \cref{section:2}, 
using the natural morphism
\[
q^{\mathrm{nc}}:\operatorname{Spec}\mathbf A_k^{\mathrm{nc}}\longrightarrow\operatorname{Spec}k,
\]
we define the corresponding obstruction subcategories of
\(\mathcal X(\mathbf A_k^{\mathrm{nc}})\), 
and similarly their composite versions.

\begin{lemma}\label{lem:basic_prop_of_Anc}
Let $\ms{X}$ be an algebraic stack over $k$.
    Then there is a natural equivalence:
\begin{align*}
    \ms{X}(\A_k)=\ms{X}(\A_k^{\mathrm{nc}})\times \prod_{v\in \Omega_k^{\mathbb C}} \ms{X}(\mathbb C),
\end{align*}
and under this equivalence:
\begin{align*}
    \ms{X}(\A_k)^{\Br}&=\ms{X}(\A_k^{\mathrm{nc}})^{\Br}\times \prod_{v\in\Omega_k^{\mathbb C}} \ms{X}(\mathbb C)\\
    \ms{X}(\A_k)^{\desc}&=\ms{X}(\A_k^{\mathrm{nc}})^{\desc}\times \prod_{v\in\Omega_k^{\mathbb C}} \ms{X}(\mathbb C)\\
    \ms{X}(\A_k)^{\et,\Br}&=\ms{X}(\A_k^{\mathrm{nc}})^{\et,\Br}\times \prod_{v\in\Omega_k^{\mathbb C}} \ms{X}(\mathbb C)\\
    \ms{X}(\A_k)^{\desc,\desc}&=\ms{X}(\A_k^{\mathrm{nc}})^{\desc,\desc}\times \prod_{v\in\Omega_k^{\mathbb C}} \ms{X}(\mathbb C)
\end{align*}
\end{lemma}

\begin{proof}
    The first two equalities follow immediately from the fact that for every linear $k$-group $G$ we have $\check \H^1_{\fppf}(\mathbb C, G_{\mathbb C})=1$ and $\Br(\mathbb C)=0$.
    
The third equality follows from the same argument as the last one, hence we only prove the last equality.
By definition,
\[
\mathcal X(\mathbf A_k)^{\mathrm{desc,desc}}
=
\bigcap_{G}
\bigcap_{f:\mathcal Y\to\mathcal X}
\bigcup_{\sigma\in \check H^1(k,G)}
f^\sigma\!\left(
\mathcal Y^\sigma(\mathbf A_k)^{\mathrm{desc}}
\right),
\]
where $G$ runs through the linear $k$-groups and
$f:\mathcal Y\to\mathcal X$ runs through the $G$-torsors.

For every $\sigma\in\check H^1(k,G)$, the twisted morphism 
$
f^\sigma:\mathcal Y^\sigma\longrightarrow\mathcal X
$
is a torsor under the inner form $G^\sigma$. At every complex
place $v$, one has 
$
\check H^1(\mathbb C,G^\sigma)=1.
$
Hence every point of $\mathcal X(\mathbb C)$ lifts to a point
of $\mathcal Y^\sigma(\mathbb C)$. Moreover, applying the
already proved decomposition for the descent obstruction to
$\mathcal Y^\sigma$, we have
\[
\mathcal Y^\sigma(\mathbf A_k)^{\mathrm{desc}}
=
\mathcal Y^\sigma(\mathbf A_k^{\mathrm{nc}})^{\mathrm{desc}}
\times
\prod_{v\in\Omega_k^{\mathbb C}}
\mathcal Y^\sigma(\mathbb C).
\]

It follows that, for every $f$ and $\sigma$,
\[
f^\sigma\!\left(
\mathcal Y^\sigma(\mathbf A_k)^{\mathrm{desc}}
\right)
=
f^\sigma\!\left(
\mathcal Y^\sigma(\mathbf A_k^{\mathrm{nc}})^{\mathrm{desc}}
\right)
\times
\prod_{v\in\Omega_k^{\mathbb C}}
\mathcal X(\mathbb C).
\]
Taking first the union over $\sigma$ and then the
intersection over all torsors $f$ gives
\[
\mathcal X(\mathbf A_k)^{\mathrm{desc,desc}}
=
\mathcal X(\mathbf A_k^{\mathrm{nc}})^{\mathrm{desc,desc}}
\times
\prod_{v\in\Omega_k^{\mathbb C}}
\mathcal X(\mathbb C).
\]

The third equality for varieties satisfying the usual hypotheses was mentioned in \cite[Equation (5-5)]{Cao_2020}.
\end{proof}

\subsection{Proof of \cref{et_br_eq_desc_2}}
We only need to prove: for  a smooth separated Deligne-Mumford stack $\ms{X}$ of finite type over $k$ with quasi-projective coarse moduli space,
we have:
\begin{align*}
        \ms{X}(\A_k)^{\et,\Br}\subseteq\ms{X}(\A_k)^{\desc,\desc}
\end{align*}

By \cref{lem:modification_of_Kre08} and its proof,
we can write $\ms{X}\simeq [Y/\GL_n]$ for a smooth quasi-projective $k$-variety $Y$.

% input---------
The morphism in \cref{lem:existence_of_C} $q_Y: Y\to C$ is $\GL_n$-invariant:
$\GL_n$ is connected, 
so its action preserves every geometrically connected component of $Y$,
i.e. its action on $C$ is trivial.
Hence it descends to
\[
\pi:\mathcal X\simeq [Y/\GL_n]\longrightarrow C,
\]
and for every $c\in C(k)$,
\[
\mathcal X_c
\simeq
[Y_c/\GL_n],
\]
where by \cref{lem:existence_of_C} $Y_c$ is smooth, quasi-projective, and geometrically integral.

For any $x\in\ms{X}(\A_k)^{\et,\Br}$,
let $x^{\mathrm{nc}}$ be the composition of morphisms of algebraic stacks $\Spec \A_k^{\mathrm{nc}}\to\Spec \A_k\xrightarrow{x} \ms{X}$.
By $\ms{X}(\A_k)^{\et,\Br}\subseteq \ms{X}(\A_k)^{\Br}$ and \cref{lem:basic_prop_of_Anc},
$x^{\mathrm{nc}}\in \ms{X}(\A_k^{\mathrm{nc}})^{\Br}$,
so $\pi\circ x^{\mathrm{nc}}\in C(\A_k^{\mathrm{nc}})^{\Br}$.
By \cite[3.3]{Liu_2015}:
\begin{align*}
    C(\A_k^{\mathrm{nc}})^{\Br}=C(k)
\end{align*}
So we have the following diagram:
\[
\begin{tikzcd}
    \Spec \A_k^{\mathrm{nc}} \arrow[r] \arrow[d,"x^{\mathrm{nc}}"] & k \arrow[d, "c"]\\
    \ms{X} \arrow[r,"\pi"] & C
\end{tikzcd}
\]

By definition,
for any finite \'etale $k$-group $G$
and any $G$-torsor $f:\ms{Y}\to \ms{X}$,
there exists $\sigma\in \check \H^1_{\fppf}(k,G)$ and $y\in \ms{Y}^{\sigma}(\A_k)^{\Br}$ such that $x=f^{\sigma}\circ y$.
Therefore we have the following diagram:
\[
\begin{tikzcd}
\Spec \mathbf{A}^{\mathrm{nc}}_k \arrow[rrd, "y^{\mathrm{nc}}" description] \arrow[rd, "\overline{y^{\mathrm{nc}}}" description] \arrow[rdd, "\overline{x^{\mathrm{nc}}}" description] \arrow[rddd] &     &      \\
& \mathcal Y^{\sigma}_c \arrow[r, "j" description] \arrow[d] & \mathcal Y^{\sigma} \arrow[d, "f^{\sigma}" description] \\
& \mathcal X_c \arrow[r, "i" description] \arrow[d]                           & \mathcal X \arrow[d]                                    \\
& k \arrow[r]                                                & C                                                      
\end{tikzcd}
\]
By \cref{lem:basic_prop_of_Anc}
$y^{\mathrm{nc}}\in \ms{Y}^{\sigma}(\A_k^{\mathrm{nc}})^{\Br}$.
Note that $c: k\to C$ is a section of a finite \'etale $k$-scheme $C$,
hence it is both an open and a closed immersion.
So there exists another open substack $\overline{\ms{Y}_c^{\sigma}}$ of $\ms{Y}^{\sigma}$ such that
\begin{align*}
    \ms{Y}^{\sigma} = \ms{Y}_c^{\sigma} \sqcup \overline{\ms{Y}_c^{\sigma}}.
\end{align*}
Hence $\Br(\ms{Y}^{\sigma})\simeq \Br(\ms{Y}_c^{\sigma})\times \Br(\overline{\ms{Y}_c^{\sigma}}) \to \Br(\ms{Y}_c^{\sigma})$ is surjective.
Then, by definition, we have $\overline{y^{\mathrm{nc}}}\in \ms{Y}_c^{\sigma}(\A_k^{\mathrm{nc}})^{\Br}$.
Similarly $\ms{X}_c$ is also an open-and-closed substack of $\ms{X}$,
hence $\H^1_{\fppf}(\ms{X},G)\to \H^1_{\fppf}(\ms{X}_c,G)$ is also surjective.
In other words,
as the $G$-torsor $\ms{Y}\to\ms{X}$ ranges over all such torsors,
all $G$-torsors on $\ms{X}_c$ can be obtained.
So $\overline{x^{\mathrm{nc}}}\in \ms{X}_c(\A_k^{\mathrm{nc}})^{\et,\Br}$.

Since \(X_c(\mathbf C)\neq\varnothing\),
\cref{lem:basic_prop_of_Anc} implies that there exists $\overline{x'}\in \ms{X}_c(\A_k)^{\et,\Br}$ such that
the composition $\Spec \A_k^{\mathrm{nc}}\to \Spec \A_k\xrightarrow{\overline{x'}} \ms{X}_c$ coincides with $\overline{x^{\mathrm{nc}}}$.
Recall that $\ms{X}_c\simeq [Y_c/\GL_n]$ for a smooth, quasi-projective, and geometrically integral $k$-variety $Y_c$.
So by \cref{et_Br_contained_in_desc_desc},
$\overline{x'}\in \ms{X}_c(\A_k)^{\desc,\desc}$.
Hence by \cref{rem_after_PGL_F_contained_in_Br_F} $x':=i \circ \overline{x'}\in \ms{X}(\A_k)^{\desc,\desc}$.

Finally, it is immediate that the composition $\Spec \A_k^{\mathrm{nc}}\to \Spec \A_k\xrightarrow{x'} \ms{X}$ coincides with $x^{\mathrm{nc}}$.
So by \cref{lem:basic_prop_of_Anc},
$x\in\ms{X}(\A_k)^{\desc,\desc}$,
i.e.
\begin{align*}
    \ms{X}(\A_k)^{\et,\Br}\subseteq \ms{X}(\A_k)^{\desc,\desc}.
\end{align*}
Combine this with \cref{et_br_eq_desc}:
\begin{align*}
    \ms{X}(\A_k)^{\desc}\subseteq \ms{X}(\A_k)^{\et,\Br}\subseteq \ms{X}(\A_k)^{\desc,\desc}\subseteq \ms{X}(\A_k)^{\desc}.
\end{align*}
So all inclusions are equalities.
\qed

\section{Comparing Brauer obstruction and the iterated Brauer obstruction}

In this section, we compare the Brauer-Manin obstruction with the composite $\Br,\Br$ obstruction on an algebraic stack $\ms{X}$.

First, we recall the definitions:
$$
\ms{X}(\A_k)^{\Br} := \ms{X}(\A_k)^{\H^2_{\et}(-,\G_m)},
$$
$$
\ms{X}(\A_k)^{\Br,\Br} := \ms{X}(\A_k)^{\H^2_{\et}(-,\G_m), \Br} = \bigcap_{g \in \H^2_{\et}(\ms{X}, \G_m)} \bigcup_{\tau \in \H^2_{\et}(k, \G_m)} g^\tau(\ms{Z}^\tau(\A_k)^{\Br}),
$$

To compare these two obstruction sets, our main input is the surjectivity of the pullback map of Brauer groups along a $\G_m$-gerbe. In \cite[Theorem 4.1.2]{SHIN-THESIS2019}, Shin proved that for a quasi-compact scheme $S$ and a torsion $\G_m$-gerbe $\pi_{\ms{G}}: \ms{G} \to S$, we have an exact sequence
$$
\H^0_{\et}(S, \Z) \to \Br'(S) \xrightarrow{\pi_{\ms{G}}^*} \Br'(\ms{G}) \to 0
$$
where $\Br'(S)$ is the cohomological Brauer group ${{\H}^2_{\et}(S,\G_m)}_{\text{tors}}$. We generalize this result to certain Artin stacks.

\medskip

We first calculate the pushforward sheaf $\pi_* \mathbb{G}_{m,\mc{Y}}$,
which will be used in the Leray spectral sequence. 

\begin{lemma}\label{lemma:gerbe-pushforward-O}
Let $\mc{X}$ be an algebraic stack over $S$, $G \to S$ an fppf qcqs group scheme, and $\pi: \mc{Y} \to \mc{X}$ a $G$-gerbe. Then $\pi$ is qcqs and the canonical map of quasi-coherent $\mc{O}_{\mc{X}}$-modules
\[
\mc{O}_{\mc{X}} \longrightarrow \pi_* \mc{O}_{\mc{Y}}
\]
is an isomorphism.
\end{lemma}

\begin{proof}
By the definition of a $G$-gerbe, there exists an $\fppf$ cover $\{p_i: U_i \to \mc{X}\}$, where $U_i = \Spec A_i$ is an affine scheme, such that the pullback stack $\mc{Y}_{U_i} := U_i \times_{\mc{X}} \mc{Y}$ is isomorphic over $U_i$ to the trivial $G$-gerbe $B G_{U_i} = [U_i/G_{U_i}]$. We form the 2-Cartesian diagram:
\[
\begin{tikzcd}
\mc{Y}_{U_i} \arrow[r, "p_{\mc{Y}}"] \arrow[d, "\pi_{U_i}"'] & \mc{Y} \arrow[d, "\pi"] \\
U_i \arrow[r, "p"'] & \mc{X}
\end{tikzcd}
\]
Then $\pi$ is qcqs since $\pi_{U_i}$ is. By \cite[Proposition 7.1.14]{alper}, there is a canonical isomorphism of quasi-coherent $\mc{O}_{U_i}$-modules:
\[
p^* (\pi_* \mc{O}_{\mc{Y}}) \cong \pi_{U_i *} (p_{\mc{Y}}^* \mc{O}_{\mc{Y}}) = \pi_{U_i *} \mc{O}_{\mc{Y}_{U_i}}.
\]
It suffices to show $\mc O_{U_i} \cong \pi_{U_i *} \mc{O}_{\mc{Y}_{U_i}}$ for all $i$. We omit the subscript $i$ from now on. 

Now $\mc{Y}_U \cong [U/G_U]$, where $G_U$ acts trivially on $U = \Spec A$. 
The category of quasi-coherent sheaves on $[U/G_U]$ is equivalent to the category of $G_U$-equivariant quasi-coherent $\mc{O}_U$-modules. Under this equivalence:
\begin{enumerate}
    \item The structure sheaf $\mc{O}_{\mc{Y}_U}$ corresponds to the $A$-module $A$ equipped with the trivial $G_U$-action.
    \item The pushforward $\pi_{U *} \mc{O}_{\mc{Y}_U}$ on $U = \Spec A$ is the quasi-coherent sheaf associated to the $A$-module of $G_U$-invariants:
    \[
    \Gamma(U, \pi_{U *} \mc{O}_{\mc{Y}_U}) = \Gamma(\mc{Y}_U, \mc{O}_{\mc{Y}_U}) = A^{G_U}=A.
    \]
\end{enumerate}
\end{proof}

\begin{proposition}\label{prop:R0-Gm-isomorphism}
Let $\mc{X}$ be an algebraic stack over $S$, let $G \to S$ be an fppf affine group scheme, and $\pi: \mc{Y} \to \mc{X}$ a $G$-gerbe. Then there is a canonical isomorphism of sheaves of abelian groups on $\mc{X}_{\fppf}$ (and hence on $\mc{X}_{\et}$):
\[
\pi_* \mathbb{G}_{m,\mc{Y}} \cong \mathbb{G}_{m,\mc{X}}.
\]
\end{proposition}

\begin{proof}
Let $V \to \mc{X}$ be an arbitrary test scheme in $(\text{Sch}/\mc{X})_{\fppf}$. Evaluating the pushforward sheaf $\pi_* \mathbb{G}_{m,\mc{Y}}$ at $V$ yields:
\[
(\pi_* \mathbb{G}_{m,\mc{Y}})(V) = \mathbb{G}_{m,\mc{Y}}(V \times_{\mc{X}} \mc{Y}) = \mr{H}^0(V \times_{\mc{X}} \mc{Y}, \mc{O}_{V \times_{\mc{X}} \mc{Y}})^\times.
\]
Let $\mc{Y}_V := V \times_{\mc{X}} \mc{Y}$, and let $\pi_V: \mc{Y}_V \to V$ be the projection morphism. Since the property of being a $G$-gerbe is stable under arbitrary base change, $\pi_V: \mc{Y}_V \to V$ is a $G_V$-gerbe over the scheme $V$.

Applying \cref{lemma:gerbe-pushforward-O} to the $G_V$-gerbe $\pi_V: \mc{Y}_V \to V$, we obtain a canonical isomorphism of $\mc{O}_V$-modules 
$\pi_{V *} \mc{O}_{\mc{Y}_V} \cong \mc{O}_V.$

Taking global sections over $V$, this gives an isomorphism of commutative rings:
\[
\mr{H}^0(\mc{Y}_V, \mc{O}_{\mc{Y}_V}) = \Gamma(V, \pi_{V *} \mc{O}_{\mc{Y}_V}) \cong \Gamma(V, \mc{O}_V) = \mr{H}^0(V, \mc{O}_V).
\]
Taking the group of units yields an isomorphism of abelian groups:
\[
\mr{H}^0(\mc{Y}_V, \mc{O}_{\mc{Y}_V})^\times \cong \mr{H}^0(V, \mc{O}_V)^\times = \mathbb{G}_{m,\mc{X}}(V).
\]
Since this isomorphism is functorial with respect to $V \to \mc{X}$, it defines an isomorphism of sheaves 
$\pi_* \mathbb{G}_{m,\mc{Y}} \cong \mathbb{G}_{m,\mc{X}}.$
\end{proof}
\begin{proposition}\label{prop:gerbe_keeps_situation}
Let $\mc{X}$ be an algebraic stack satisfying \Cref{sit:setup}. Let $\pi \colon \mc{Y} \to \mc{X}$ be a $\mathbb{G}_m$-gerbe over $\mc{X}$. Then $\mc{Y}$ also satisfies \Cref{sit:setup}.
\end{proposition}

\begin{proof}
We verify each requirement of \Cref{sit:setup} for the stack $\mc{Y}$:
\begin{enumerate}
    \item[\textnormal{(i)}] \textit{Qcqs and Noetherian}: Since $\mathbb{G}_m$ is a qcqs group scheme of finite type over $\mathbb{Z}$, the gerbe projection $\pi \colon \mc{Y} \to \mc{X}$ is a qcqs morphism of finite type; hence $\mc Y$ is qcqs and Noetherian.
    
    \item[\textnormal{(ii)}] \textit{Regularity}: Since $\mathbb{G}_m$ is a smooth algebraic group, the gerbe morphism $\pi \colon \mc{Y} \to \mc{X}$ is smooth. The smoothness of $\pi$ combined with the regularity of $\mc{X}$ implies that $\mc{Y}$ is regular.
    
    \item[\textnormal{(iii)}] \textit{Generically affine stabilizer}: Let $y \in \mc{Y}$ be a generic point lying over $x = \pi(y) \in \mc{X}$. 
    Since $\pi$ is a gerbe,
    $|\ms Y|\to |\ms X|$ is a homeomorphism by \cite[\href{https://stacks.math.columbia.edu/tag/06R9}{Tag 06R9}]{SP},
    hence $x$ is a generic point of $\ms X$.
    \iffalse
    Let $K = \kappa(y) \cong \kappa(x)$ be their common residue field, and let $K^{\sep}$ be a separable closure of $K$. 
    \fi 
    Choose a separably closed field $\Omega$ and corresponding compatible separable geometric points $\bar{y} \colon \Spec \Omega \to \mc{Y}$  and $\bar{x} \colon \Spec \Omega \to \mc{X}$. By the definition of a $\mathbb{G}_m$-gerbe, there is a short exact sequence of algebraic group schemes over $\Omega$:
    \[
    1 \longrightarrow \mathbb{G}_{m, \Omega} \longrightarrow \operatorname{Aut}_{\mc{Y}}(\bar{y}) \xrightarrow{\;\psi\;} \operatorname{Aut}_{\mc{X}}(\bar{x}) \longrightarrow 1,
    \]
    where the stabilizer group space $\operatorname{Aut}_{\mc{X}}(\bar{x})$ is an affine algebraic group over $\Omega$. 
    
    The quotient morphism $\psi \colon \operatorname{Aut}_{\mc{Y}}(\bar{y}) \to \operatorname{Aut}_{\mc{X}}(\bar{x})$ is a faithfully flat $\mathbb{G}_{m, \Omega}$-torsor in the fppf topology. Since $\mathbb{G}_{m, \Omega}$ is affine, $\psi$ is an affine morphism. This implies that $\operatorname{Aut}_{\mc{Y}}(\bar{y})$ is an affine algebraic group over $\Omega$. Thus, $\mc{Y}$ has generically affine stabilizer.
\end{enumerate}
\end{proof}

\begin{theorem} \label{thm:surjective_brauer_gerbe}
Let $\ms{X}$ be a stack in \cref{sit:setup}. Let $\pi: \ms{Y} \to \ms{X}$ be a $\G_m$-gerbe over $\ms{X}$. Then $\Br(\ms{X})$ and $\Br(\ms{Y})$ are torsion and we have an exact sequence: 
$$
\H_{\et}^0(\ms X, \mathbb Z) \longrightarrow \Br(\ms{X}) \xrightarrow{\quad \pi^* \quad} \Br(\ms{Y}) \longrightarrow 0
$$
\end{theorem}

\begin{proof}
We compute the Brauer group using the Leray spectral sequence associated to the morphism $\pi: \ms{Y} \to \ms{X}$ and the sheaf $\G_{m,\ms{Y}}$: 
$$
E_2^{p,q} = \H^p_{\et}(\ms{X}, R^q\pi_* \G_{m,\ms{Y}}) \implies \H^{p+q}_{\et}(\ms{Y}, \G_{m,\ms{Y}}).
$$
To analyze the stalks of $R^q\pi_* \G_{m,\ms{Y}}$, we choose a smooth atlas $p: U \to \ms{X}$ where $U$ is a scheme. By the base change theorem \cite[\href{https://stacks.math.columbia.edu/tag/075H}{Tag 075H}]{SP},
$$
p^{-1} R^q\pi_* \G_{m,\ms{Y}} \simeq R^q\pi_{U*} \G_{m,\ms{Y}_U},
$$
where $\ms{Y}_U := \ms{Y} \times_{\ms{X}} U$ and $\pi_U: \ms{Y}_U \to U$ is the base change of $\pi$. Since $U$ is a scheme, by the argument of \cite[Section 4.5, in particular Lemma 4.5.2]{SHIN-THESIS2019}, we have $R^2\pi_{U*} \G_{m,\ms{Y}_U} = 0$. This implies that the sheaf $R^2\pi_* \G_{m,\ms{Y}}$ vanishes on $\ms{X}$.

Furthermore, by \cite[Lemma 4.2.2]{SHIN-THESIS2019}, we have $R^1\pi_* \G_{m,\ms{Y}} \simeq \underline{\Z}$. We also have $R^0\pi_* \G_{m,\ms{Y}} = \G_{m,\ms{X}}$ by \cref{prop:R0-Gm-isomorphism}. Thus, the seven-term exact sequence yields
$$
\H^0_{\et}(\ms{X}, \underline{\Z}) \longrightarrow \H^2_{\et}(\ms{X}, \G_{m,\ms{X}}) \xrightarrow{\pi^*} \H^2_{\et}(\ms{Y}, \G_{m,\ms{Y}}) \longrightarrow \H^1_{\et}(\ms{X}, \underline{\Z}).
$$
Applying \cite[Lemma A.0.2]{SHIN-THESIS2019} to the short exact sequence $0 \to \underline{\mathbb{Z}} \xrightarrow{n} \underline{\mathbb{Z}} \to \underline{\mathbb{Z}/n\mathbb{Z}} \to 0$, we obtain that multiplication by \(n\) on \(\H^1_{\et}(X,\underline{\mathbb Z})\) is injective.
\iffalse
For any positive integer $n$,
from the short exact sequence $0 \to \underline{\mathbb{Z}} \xrightarrow{n} \underline{\mathbb{Z}} \to \underline{\mathbb{Z}/n\mathbb{Z}} \to 0$, by \cite[Lemma A.0.2]{SHIN-THESIS2019},
we have $\H^1_{\et}(\ms{X}, \underline{\mathbb{Z}})\xrightarrow{\cdot n} \H^1_{\et}(\ms{X}, \underline{\mathbb{Z}})$ is injective. 
\fi
This implies that $\H^1_{\et}(\ms{X}, \underline{\mathbb{Z}})$ is torsion-free.

Now, for any class $\beta \in \Br(\ms{Y})$, since $\Br(\ms{Y})$ is a torsion group (\cref{prop:gerbe_keeps_situation}, \cref{thm:br_torsion}), there exists an integer $n \ge 1$ such that $n\beta = 0$. The image of $\beta$ under the boundary map $\H^2_{\et}(\ms{Y}, \G_{m,\ms{Y}}) \to \H^1_{\et}(\ms{X}, \underline{\Z})$ is therefore $n$-torsion. Since $\H^1_{\et}(\ms{X}, \underline{\mathbb{Z}})$ is torsion-free, this image must be zero. By exactness, $\pi^*$ is surjective.
\end{proof}

With this theorem, we can now establish the equality between the Brauer-Manin obstruction and the composite $\Br,\Br$ obstruction.

\begin{theorem}\label{Br_Br_eq_Br}
Let $\ms{X}$ be a stack of finite type over $k$ satisfying \cref{sit:setup}. Then
$$
\ms{X}(\A_k)^{\Br} = \ms{X}(\A_k)^{\Br,\Br}.
$$
\end{theorem}

\begin{proof}
The inclusion $\ms{X}(\A_k)^{\Br,\Br} \subseteq \ms{X}(\A_k)^{\Br}$ is straightforward; it suffices to prove the reverse inclusion.

Under our assumptions on $\ms{X}$, $\ms{X}$ is quasi-compact and its Brauer group is a torsion group (\cref{thm:br_torsion}). Let $x \in \ms{X}(\A_k)^{\Br}$. We want to show that for any $\G_m$-gerbe $g: \ms{Z} \to \ms{X}$, there exists a twist $\tau \in \H^2_{\et}(k, \G_m)$ and a point $z \in \ms{Z}^\tau(\A_k)^{\Br}$ such that $g^\tau(z) = x$. 

Since $x \in \ms{X}(\A_k)^{\Br}$, the standard descent formalism for gerbes ensures that $x$ lifts to some adelic point $z \in \ms{Z}^\tau(\A_k)$ on a twisted gerbe $g^\tau: \ms{Z}^\tau \to \ms{X}$ for some $\tau \in \H^2_{\et}(k, \G_m)$.
We next verify that $z \in \ms{Z}^\tau(\A_k)^{\Br}$.

By \cref{thm:surjective_brauer_gerbe}, the pullback map $(g^\tau)^*: \Br(\ms{X}) \to \Br(\ms{Z}^\tau)$ is surjective. 
Thus, the image of $\Br(\ms Z^\tau) \to \Br(\A_k)$ lies in the image of $\Br(\mc X) \to \Br(\A_k)$, which lies in the image of $\Br(k) \to \Br(\A_k)$. We conclude that $z \in \ms{Z}^\tau(\A_k)^{\Br}$. Hence $x \in \ms{X}(\A_k)^{\Br,\Br}$,
which completes the proof.
\end{proof}

Moreover,
we have the following interesting consequence:
\begin{corollary}
    For every smooth separated Deligne-Mumford stack $\mathcal{X}$ of finite type over $k$ with quasi-projective coarse moduli space,
    for every positive integer $n$ and $\mathrm{obs}_i\in\{\desc,\Br\}$ for $i\in[1,n]$,
    \begin{align*}
        \ms{X}(\A_k)^{\mathrm{obs}_1,\ldots,\mathrm{obs}_n}=
        \left\{
        \begin{aligned}
            &\ms{X}(\A_k)^{\desc} \text{ if there exists $i$ such that $\mathrm{obs}_i=\desc$;}\\
            &\ms{X}(\A_k)^{\Br}  \text{ if }\mathrm{obs}_1=\ldots=\mathrm{obs}_n=\Br.
        \end{aligned}
        \right.
    \end{align*}
\end{corollary}

\begin{proof}
    We prove by induction.
    The case $n=1$ is trivial.
    Assume that the statement holds for \(n=m\). We prove it for \(n=m+1\).
    \begin{enumerate}
        \item If one of $\mathrm{obs}_2,\ldots,\mathrm{obs}_{m+1}$ is $\desc$, then by induction:
        \begin{align*}
        \ms X(\A_k)^{\desc}
        &= \ms X(\A_k)^{\desc,\desc} &\text{ \cref{et_br_eq_desc_2}}\\
&=\ms X(\A_k)^{\desc,\mathrm{obs}_2,\ldots,\mathrm{obs}_{m+1}} &\text{ by induction}\\
&\subseteq \ms X(\A_k)^{\Br,\mathrm{obs}_2,\ldots,\mathrm{obs}_{m+1}} &\text{ \cref{rem_after_PGL_F_contained_in_Br_F}}\\
&\subseteq \ms X(\A_k)^{\mathrm{obs}_2,\ldots,\mathrm{obs}_{m+1}} &\text{ \cref{rem_after_PGL_F_contained_in_Br_F}}\\
&=\ms X(\A_k)^{\desc} &\text{ by induction.}
        \end{align*}
        Hence $\ms X(\A_k)^{\mathrm{obs}_1,\mathrm{obs}_2,\ldots,\mathrm{obs}_{m+1}}=\ms X(\A_k)^{\desc}$ in this case. 
        \item If $\mathrm{obs}_2=\ldots =\mathrm{obs}_{m+1}=\Br$,
        then by induction:
        \begin{align*}
            \ms X(\A_k)^{\desc,\Br,\ldots,\Br}=\ms X(\A_k)^{\desc,\Br}=\ms X(\A_k)^{\desc},
        \end{align*}
        and by the same argument as in the proof of \cref{Br_Br_eq_Br}:
        \begin{align*}
            \ms X(\A_k)^{\Br,\ldots,\Br}=\ms X(\A_k)^{\Br}.
        \end{align*}
    \end{enumerate}
    So the induction holds.
\end{proof}

\iffalse
\begin{remark}
    For every smooth separated Deligne-Mumford stack $\mathcal{X}$ of finite type over $k$ with quasi-projective coarse moduli space,
    $\ms{X}(\A_k)^{\Br}=\ms{X}(\A_k)^{\Br,\Br}$ as in \cref{Br_Br_eq_Br},
    $\ms{X}(\A_k)^{\desc}=\ms{X}(\A_k)^{\desc,\desc}$ as in \cref{et_br_eq_desc_2}.
    Notice that by \cref{et_br_eq_desc_2}, \cref{lem:torsor_keeps_good_condition}
    and \cref{desc_contained_in_Br}:
    \begin{align*}
        \ms{X}(\A_k)^{\desc}
        =\ms{X}(\A_k)^{\desc,\desc}
        \subseteq \ms{X}(\A_k)^{\desc,\Br}
        \subseteq \ms{X}(\A_k)^{\desc},
    \end{align*}
    and by \cref{PGL_F_contained_in_Br_F},
    \begin{align*}
        \ms{X}(\A_k)^{\desc}
        =\ms{X}(\A_k)^{\desc,\desc}
        \subseteq \ms{X}(\A_k)^{\PGL,\desc}
        \subseteq \ms{X}(\A_k)^{\Br,\desc}
        \subseteq \ms{X}(\A_k)^{\desc}.
    \end{align*}
    So $\ms{X}(\A_k)^{\desc}=\ms{X}(\A_k)^{\desc,\Br}$ and $\ms{X}(\A_k)^{\desc}=\ms{X}(\A_k)^{\Br,\desc}$.

    Moreover we have the following interesting consequence:
    Let $\ms{X}$ be as above. For every positive integer $n$ and $\mathrm{obs}_i\in\{\desc,\Br\}$ for $i\in[1,n]$,
    \begin{align*}
        \ms{X}(\A_k)^{\mathrm{obs}_1,\ldots,\mathrm{obs}_n}=
        \left\{
        \begin{aligned}
            &\ms{X}(\A_k)^{\desc} \quad \exists i\text{ $\mathrm{obs}_i=\desc$;}\\
            &\ms{X}(\A_k)^{\Br} \quad  \mathrm{obs}_1=\ldots=\mathrm{obs}_n=\Br.
        \end{aligned}
        \right.
    \end{align*}
    which answers the question in \cite{LW26}.
\end{remark}
\fi

\begin{appendix}

\section{Azumaya algebras and Brauer groups on algebraic stacks}
In this section we will refer to the notation and definitions in \cite[Section 1]{SHIN-THESIS2019}.

Let $\ms{X}$ be an algebraic stack over a field $k$. 
Let $\ms{X}_{\sm}$ denote the big smooth site on $\ms{X}$, 
and let $\mc{O}_{\ms{X}}$ be its structure sheaf. 

\begin{definition}[Azumaya algebra]
    An \textcolor{blue}{Azumaya algebra} $\mathscr{A}$ on $\ms{X}$ is a 
    quasi-coherent, unital, associative 
    $\mathscr{O}_{\ms{X}}$-algebra such that there exists a smooth covering $\mathfrak{U} = \{U_i \to \ms{X}\}_{i \in I}$ such that for each \(i\), there is a positive integer \(n_i\) and an \(\mathcal O_{U_i}\)-algebra isomorphism:
    \begin{align*}
        \alpha_i : \mc{A}|_{U_i} \xrightarrow{\sim} \mathrm{Mat}_{n_i \times n_i}(\mc{O}_{U_i}).
    \end{align*}
    If $n_i = n$ for all $i$, we say $\mc{A}$ is an Azumaya algebra of constant rank $n^2$.
\end{definition}

\begin{remark}

We briefly clarify the convention concerning \v{C}ech cohomology used below. 
Let $\ms X_{\mathrm{sm}}^{\mathrm{st}}$ 
denote the big smooth site of algebraic stacks over $\ms X$, 
whose coverings are jointly surjective families of representable smooth morphisms. The natural inclusion
\[
\ms X_{\mathrm{sm}}=(\operatorname{Sch}/\ms X)_{\mathrm{sm}}
\longrightarrow
\ms X_{\mathrm{sm}}^{\mathrm{st}}
\]
induces an equivalence of the associated topoi. 

Unlike \(\ms X_{\mathrm{sm}}\), 
the site \(\ms X_{\mathrm{sm}}^{\mathrm{st}}\) has \(\ms X\) as a final object. 
We therefore define the \v{C}ech cohomology $\check{\H}^1_{\mathrm{sm}}(\ms X,\ms F)$ for a sheaf of groups $\ms F$ and $\check{\H}^i_{\mathrm{sm}}(\ms X,\ms G)$ for a sheaf of abelian groups $\ms G$ appearing below using coverings of the final object $\ms X$ in $\ms X_{\mathrm{sm}}^{\mathrm{st}}$. 
Moreover, every such smooth covering can be refined by a smooth covering
\[
\{U_i\to \ms X\}_{i\in I}
\]
with each $U_i$ a scheme. 
Hence the colimit defining $\check{\H}^1_{\mathrm{sm}}(\ms X,\ms F)$ and $\check{\H}^i_{\mathrm{sm}}(\ms X,\ms G)$ may equivalently be taken over such scheme-valued smooth coverings.
\end{remark}

\subsection{Induced classes in \v{C}ech cohomology}
Let $\mc{A}$ be an Azumaya algebra of constant rank $n^2$ on $\ms{X}$. We now explain how $\mc{A}$ naturally induces an element in the \v{C}ech cohomology set $\check{\H}^1_{\sm}(\ms{X}, \PGL_n)$. 

\begin{pg}
    Choosing a trivializing smooth covering $\mathfrak{U} = \{U_i \to \ms{X}\}$ and isomorphisms $\alpha_i$ as above, we obtain transition functions on the pairwise intersections $U_{ij} := U_i \times_{\ms{X}} U_j$:
\begin{align*}
    \beta_{ij} := (\alpha_i|_{U_{ij}}) \circ (\alpha_j|_{U_{ij}})^{-1} : \mathrm{Mat}_{n \times n}(\mc{O}_{U_{ij}}) \xrightarrow{\sim} \mathrm{Mat}_{n \times n}(\mc{O}_{U_{ij}}).
\end{align*}
Clearly, $\beta_{ij}$ is an $\mc{O}_{U_{ij}}$-algebra automorphism. By the Skolem-Noether theorem for locally ringed topoi (cf. \cite[1.1.8]{SHIN-THESIS2019}), 
the canonical conjugation morphism
\begin{align*}
    \PGL_n(\mc{O}_{U_{ij}}) \xrightarrow{\sim} \underline{\mathrm{Aut}}_{\mc{O}\text{-alg}}(\mathrm{Mat}_{n \times n}(\mc{O}_{U_{ij}}))
\end{align*}
is an isomorphism. Under this identification, the collection of automorphisms $\{\beta_{ij}\}_{i,j \in I}$ corresponds to a unique collection of sections $g_{ij} \in \PGL_n(U_{ij})$. Since $\alpha_i \circ \alpha_j^{-1} \circ \alpha_j \circ \alpha_k^{-1} = \alpha_i \circ \alpha_k^{-1}$ on triple intersections $U_{ijk}$, the sections satisfy the cocycle condition $g_{ij}g_{jk} = g_{ik}$. Therefore, $\{g_{ij}\}$ defines a \v{C}ech 1-cocycle, which yields a well-defined class 
\begin{align*}
    [\mc{A}] \in \check{\H}^1_{\sm}(\ms{X}, \PGL_n).
\end{align*}
This shows that isomorphism classes of Azumaya algebras of rank $n^2$ are classified by $\check{\H}^1_{\sm}(\ms{X}, \PGL_n)$.
\end{pg}

\subsection{The gerbe of trivializations}
One can associate a $\G_m$-gerbe to any Azumaya algebra $\mc{A}$, 
which provides a coordinate-free approach to the cohomological Brauer group.
\begin{definition}[Gerbe of trivializations]
    Let $\mc{A}$ be an Azumaya algebra on $\ms{X}$. The 
    \textcolor{blue}{gerbe of trivializations} of $\mc{A}$, 
    denoted by $\ms{G}_{\mc{A}}$, is the category fibered in groupoids over $\ms{X}_{\sm}$ defined as follows:
    \begin{itemize}
        \item \textbf{Objects:} A triple $(U, \mc{E}, \sigma)$, where $U \in \ms{X}_{\sm}$, $\mc{E}$ is a finite locally free $\mc{O}_U$-module of everywhere positive rank, and $\sigma: \underline{\mathrm{End}}_{\mc{O}_U}(\mc{E}) \xrightarrow{\sim} \mc{A}|_U$ is an isomorphism of $\mc{O}_U$-algebras.
        \item \textbf{Morphisms:} A morphism $(U_1, \mc{E}_1, \sigma_1) \to (U_2, \mc{E}_2, \sigma_2)$ consists of a morphism $f: U_1 \to U_2$ in $\ms{X}_{\sm}$ and an isomorphism $f^\sharp: f^*\mc{E}_2 \xrightarrow{\sim} \mc{E}_1$ of $\mc{O}_{U_1}$-modules such that $\sigma_2$ is compatible with $\sigma_1$ via the conjugation induced by $f^\sharp$.
    \end{itemize}
\end{definition}

\begin{pg}\label{azumaya_connecting_homomorphism}
Consider the short exact sequence of sheaves of groups on $\ms{X}_{\sm}$:
\begin{align}\label{eq:ses_groups}
    1 \to \G_m \to \GL_n \to \PGL_n \to 1.
\end{align}
This sequence induces a connecting homomorphism in \v{C}ech cohomology:
\begin{align*}
    \delta : \check{\H}^1_{\sm}(\ms{X}, \PGL_n) \to \check{\H}^2_{\sm}(\ms{X}, \G_m).
\end{align*}

Explicitly, given the class $[\mc{A}] \in \check{\H}^1_{\sm}(\ms{X}, \PGL_n)$ represented by the 1-cocycle $\{g_{ij}\}$ with $g_{ij} \in \PGL_n(U_{ij})$, one can (after possibly refining the covering $\mathfrak{U}$) lift $g_{ij}$ to sections $\tilde{g}_{ij} \in \GL_n(U_{ij})$. The failure of $\{\tilde{g}_{ij}\}$ to be a 1-cocycle is measured by
\begin{align*}
    c_{ijk} := \tilde{g}_{ij} \tilde{g}_{jk} \tilde{g}_{ik}^{-1} \in \GL_n(U_{ijk}).
\end{align*}
Since $g_{ij}g_{jk}g_{ik}^{-1} = 1$ in $\PGL_n$, the element $c_{ijk}$ takes values in the kernel of the projection, namely the center $\G_m(U_{ijk})$. One can verify that $\{c_{ijk}\}$ satisfies the \v{C}ech 2-cocycle condition, thus defining a class $\delta([\mc{A}]) \in \check{\H}^2_{\sm}(\ms{X}, \G_m)$.
\end{pg}

\begin{lemma}\label{lem:gerbe_class_coincides}
    Let $\mc{A}$ be an Azumaya algebra of rank $n^2$ on $\ms{X}$. 
    Then by the natural homomorphism:
    \begin{align}\label{natural_cech_H2_to_H2}
        \check{\H}^2_{\sm}(\ms{X}, \G_m) \rightarrow {\H}^2_{\sm}(\ms{X}, \G_m), 
    \end{align}
    and the bijection (cf. \cite[12.2.8]{OLSSON})
    \begin{align*}
        {\H}^2_{\sm}(\ms{X}, \G_m) \xrightarrow{\sim} \{\text{isomorphism classes of $\G_m$-gerbes over $\ms{X}$}\},
    \end{align*}
    the \v{C}ech  2-cocycle $\delta([\mc{A}])$
    maps to the 2-cocycle defining the class of the gerbe $[\ms{G}_{\mc{A}}] \in {\H}^2_{\sm}(\ms{X}, \G_m)$.
\end{lemma}

\begin{proof}
    We recall the standard procedure for extracting a \v{C}ech 2-cocycle from a $\G_m$-gerbe $\ms{G}$. One chooses a smooth covering $\mathfrak{U} = \{U_i \to \ms{X}\}$ such that there exist local objects $x_i \in \ms{G}(U_i)$. On the intersections $U_{ij}$, one chooses isomorphisms $\phi_{ij} : x_j|_{U_{ij}} \xrightarrow{\sim} x_i|_{U_{ij}}$. On the triple intersections $U_{ijk}$, the composition 
    \begin{align*}
        c_{ijk} := \phi_{ij} \circ \phi_{jk} \circ \phi_{ik}^{-1}
    \end{align*}
    is an automorphism of $x_i|_{U_{ijk}}$. Since $\ms{G}$ is a $\G_m$-gerbe, $\mathrm{Aut}(x_i) = \G_m(U_{ijk})$, so $c_{ijk} \in \G_m(U_{ijk})$. 
    By the homomorphism \eqref{natural_cech_H2_to_H2}, 
    the collection $\{c_{ijk}\}$ forms a \v{C}ech 2-cocycle whose image is \([\ms G]\).

    We now apply this procedure to the gerbe of trivializations $\ms{G}_{\mc{A}}$. 
    Choose a covering $\mathfrak{U} = \{U_i \to \ms{X}\}$ trivializing $\mc{A}$, with isomorphisms $\alpha_i : \mc{A}|_{U_i} \xrightarrow{\sim} \mathrm{Mat}_{n \times n}(\mc{O}_{U_i})$. 
    The local objects in $\ms{G}_{\mc{A}}(U_i)$ are naturally given by 
    \begin{align*}
        x_i := (U_i, \mc{O}_{U_i}^{\oplus n}, \alpha_i^{-1}).
    \end{align*}
    An isomorphism $\phi_{ij} : x_j \xrightarrow{\sim} x_i$ over $U_{ij}$ is, by definition, an isomorphism of the underlying vector bundles $\mc{O}_{U_{ij}}^{\oplus n} \xrightarrow{\sim} \mc{O}_{U_{ij}}^{\oplus n}$ (which is given by a matrix $\tilde{g}_{ij} \in \GL_n(U_{ij})$) such that it intertwines the algebra isomorphisms:
    \begin{align*}
        \alpha_i^{-1} \circ \mathrm{Ad}(\tilde{g}_{ij}) = \alpha_j^{-1} \implies \mathrm{Ad}(\tilde{g}_{ij}) = \alpha_i \circ \alpha_j^{-1} = \beta_{ij}.
    \end{align*}
    This means that the local isomorphism $\phi_{ij}$ in the gerbe $\ms{G}_{\mc{A}}$ is precisely a lift $\tilde{g}_{ij} \in \GL_n(U_{ij})$ of the transition function $\beta_{ij} \in \PGL_n(U_{ij})$ representing the class $[\mc{A}] \in \check{\H}^1_{\sm}(\ms{X}, \PGL_n)$.

    Finally, the obstruction to the cocycle condition in the gerbe is the composition of these morphisms:
    \begin{align*}
        c_{ijk} = \phi_{ij} \circ \phi_{jk} \circ \phi_{ik}^{-1}.
    \end{align*}
    Since morphisms in $\ms{G}_{\mc{A}}$ are composed by multiplying the underlying matrices, this composition is exactly
    \begin{align*}
        c_{ijk} = \tilde{g}_{ij} \tilde{g}_{jk} \tilde{g}_{ik}^{-1} \in \GL_n(U_{ijk}).
    \end{align*}
    Because $\mathrm{Ad}(c_{ijk}) = \beta_{ij}\beta_{jk}\beta_{ik}^{-1} = \mathrm{id}$, the matrix $c_{ijk}$ lies in the center of $\GL_n$, which is $\G_m(U_{ijk})$. 

    Comparing with the discussion in \Cref{azumaya_connecting_homomorphism},
    this $\{c_{ijk}\}$ is precisely the 2-cocycle obtained by applying the connecting homomorphism $\delta$ to the 1-cocycle $\{\beta_{ij}\}$. Thus, the two classes coincide in $\check{\H}^2_{\sm}(\ms{X}, \G_m)$.
\end{proof}

\begin{definition}[Brauer equivalence and the Brauer group]
    Let $\mc{A}_1$ and $\mc{A}_2$ be two Azumaya algebras on an algebraic stack $\ms{X}$. 
    We say $\mc{A}_1$ and $\mc{A}_2$ are \textcolor{blue}{Brauer equivalent} 
    if there exist finite locally free $\mc{O}_{\ms{X}}$-modules $\mc{E}_1$ and $\mc{E}_2$ of everywhere positive rank, and an isomorphism of $\mc{O}_{\ms{X}}$-algebras:
    \begin{align*}
        \mc{A}_1 \otimes_{\mc{O}_{\ms{X}}} \underline{\mathrm{End}}_{\mc{O}_{\ms{X}}}(\mc{E}_1) \simeq \mc{A}_2 \otimes_{\mc{O}_{\ms{X}}} \underline{\mathrm{End}}_{\mc{O}_{\ms{X}}}(\mc{E}_2).
    \end{align*}
    The \textcolor{blue}{Azumaya Brauer group} 
    of $\ms{X}$, 
    denoted by $\mathrm{Br}_{\mathrm{Az}}(\ms{X})$, 
    is the set of Brauer equivalence classes of Azumaya algebras on $\ms{X}$. 
    The group operation is induced by the tensor product $[\mc{A}_1] \cdot [\mc{A}_2] := [\mc{A}_1 \otimes_{\mc{O}_{\ms{X}}} \mc{A}_2]$. The identity element is the class of trivial Azumaya algebras $[\underline{\mathrm{End}}_{\mc{O}_{\ms{X}}}(\mc{E})]$, and the inverse of $[\mc{A}]$ is given by its opposite algebra $[\mc{A}^{\mathrm{op}}]$.
\end{definition}

\begin{definition}[Rank of a sheaf]
    Let $\ms{X}$ be an algebraic stack and let $\mc{F}$ be a quasi-coherent sheaf of finite type (e.g., an Azumaya algebra $\mc{A}$). For a point $x \in |\ms{X}|$, choose a representative $x_K: \Spec K \to \ms{X}$. The \textcolor{blue}{rank} of $\mc{F}$ at $x$, denoted by $\operatorname{rank}_x(\mc{F})$, is defined as
    \begin{align*}
        \operatorname{rank}_x(\mc{F}) := \dim_K(x_K^*\mc{F}).
    \end{align*}
    This integer is independent of the choice of the representative field-valued point $x_K$. Equivalently, for any smooth morphism $p: U \to \ms{X}$ from a scheme $U$ and any point $u \in U$ mapping to $x$, one has $\operatorname{rank}_x(\mc{F}) = \operatorname{dim}_{\kappa (u)}(p^*\mc{F}(u))$.
\end{definition}

\begin{lemma}\label{lem:rank_locally_constant}
    Let $\ms{X}$ be an algebraic stack, and let $\mc{A}$ be an Azumaya algebra on $\ms{X}$. Then the index function
    \begin{align*}
        \deg_{\mc{A}} : |\ms{X}| &\longrightarrow \mathbb{N}^* \\
        x &\longmapsto \sqrt{\operatorname{rank}_x(\mc{A})}
    \end{align*}
    is a well-defined, locally constant function on the topological space $|\ms{X}|$.
\end{lemma}

\begin{proof}
    First, we check that $\deg_{\mc{A}}$ is well-defined. For any point $x \in |\ms{X}|$ represented by $x_K: \Spec K \to \ms{X}$, the pullback $x_K^*\mc{A}$ is an Azumaya algebra over the field $K$, i.e., a central simple $K$-algebra. By Wedderburn-Artin theory, $\dim_K(x_K^*\mc{A}) = n^2$ for a unique $n \in \mathbb{N}^*$, so $\deg_{\mc{A}}(x) = n$ is well-defined.

    Next, we show that $\deg_{\mc{A}}$ is locally constant on $|\ms{X}|$, i.e. for each $n \in \mathbb{N}^*$, 
    \begin{align*}
        W_n := \{ x \in |\ms{X}| \mid \deg_{\mc{A}}(x) = n \} \subseteq |\ms{X}|
    \end{align*}
    is both open and closed in $|\ms{X}|$. 
    A subset $W \subseteq |\ms{X}|$ is open (resp. closed) if and only if for every smooth morphism $p: U \to \ms{X}$ from a scheme $U$, the preimage $|p|^{-1}(W)$ is open (resp. closed). Let $p: U \to \ms{X}$ be an arbitrary smooth morphism from a scheme $U$. Notice that for any $u \in |U|$ mapping to $x = p(u) \in |\ms{X}|$, choosing the representative $x_{\kappa(u)}: \Spec \kappa(u) \to U \xrightarrow{p} \ms{X}$ yields:
    \begin{align*}
        \operatorname{rank}_{p(u)}(\mc{A}) = \dim_{\kappa(u)}(x_{\kappa(u)}^*\mc{A}) = \operatorname{rank}_{\mc{O}_{U, u}}((p^*\mc{A})_u).
    \end{align*}
    $|p|^{-1}(W_n) = \{ u \in |U| \mid \operatorname{rank}_{\mc{O}_{U, u}}((p^*\mc{A})_u) = n^2 \}
    $ which is both open and closed in $|U|$.
    Since $p: U \to \ms{X}$ was chosen arbitrarily, $W_n$ is both open and closed in $|\ms{X}|$ for every $n \in \mathbb{N}^*$. Therefore, $\deg_{\mc{A}}$ is continuous when $\mathbb{N}^*$ is given the discrete topology, which completes the proof. 
\end{proof}

\begin{lemma}\label{every_class_has_constant_rank}
    If $\ms{X}$ is an algebraic stack of finite type over $k$, then every equivalence class in $\mathrm{Br}_{\mathrm{Az}}(\ms{X})$ has a representative that is of constant rank $n^2$ for some $n\in\mathbb{N}^*$.
\end{lemma}

\begin{proof}
    Let $[\mc{A}] \in \mathrm{Br}_{\mathrm{Az}}(\ms{X})$ be a Brauer class represented by an Azumaya algebra $\mc{A}$ on $\ms{X}$. The index function $\deg_{\mc{A}}$ is a locally constant function on $|\ms{X}|$.

    Since $|\ms{X}|$ is Noetherian, $|\ms{X}|$ has only finitely many connected components. $\ms{X}$ decomposes into a finite coproduct of connected open and closed substacks:
    \begin{align*}
        \ms{X} \simeq \coprod_{i=1}^m \ms{X}_i.
    \end{align*}
    Since each underlying topological space $|\ms{X}_i|$ is connected and $\deg_{\mc{A}}$ is locally constant, the restriction $\deg_{\mc{A}}|_{|\ms{X}_i|}$ is identically equal to a constant integer $n_i \in \mathbb{N}^*$ for each $i = 1, \dots, m$. Thus, $\mc{A}|_{\ms{X}_i}$ is an Azumaya algebra of constant rank $n_i^2$ over $\ms{X}_i$.

    Now set $n := \prod_{i=1}^m n_i \in \mathbb{N}^*$, and for each $i \in \{1, \dots, m\}$, define $r_i := n / n_i \in \mathbb{N}^*$. By the canonical 2-equivalence of categories of quasi-coherent modules:
    \begin{align*}
        \mathrm{QCoh}(\ms{X}) \xrightarrow{\sim} \prod_{i=1}^m \mathrm{QCoh}(\ms{X}_i), \quad \mc{F} \mapsto (\mc{F}|_{\ms{X}_1}, \dots, \mc{F}|_{\ms{X}_m}),
    \end{align*}
    we can define a quasi-coherent $\mc{O}_{\ms{X}}$-module $\mc{E} \in \mathrm{QCoh}(\ms{X})$ by specifying its components as $\mc{E}|_{\ms{X}_i} := \mc{O}_{\ms{X}_i}^{\oplus r_i}$ for $i = 1, \dots, m$. $\mc{E}$ is a finite locally free $\mc{O}_{\ms{X}}$-module of everywhere positive rank.

    Consider the tensor product Azumaya algebra on $\ms{X}$:
    \begin{align*}
        \mc{A}' := \mc{A} \otimes_{\mc{O}_{\ms{X}}} \underline{\mathrm{End}}_{\mc{O}_{\ms{X}}}(\mc{E}).
    \end{align*}
    By definition of Brauer equivalence, $\mc{A}'$ represents the same class $[\mc{A}'] = [\mc{A}]$ in $\mathrm{Br}_{\mathrm{Az}}(\ms{X})$. Restricting $\mc{A}'$ to each connected substack $\ms{X}_i$, we obtain:
    \begin{align*}
        \mc{A}'|_{\ms{X}_i} \simeq \mc{A}|_{\ms{X}_i} \otimes_{\mc{O}_{\ms{X}_i}} \underline{\mathrm{End}}_{\mc{O}_{\ms{X}_i}}(\mc{O}_{\ms{X}_i}^{\oplus r_i}).
    \end{align*}
    Therefore, the index function $\deg_{\mc{A}'}$ on $|\ms{X}_i|$ satisfies:
    \begin{align*}
        \deg_{\mc{A}'}|_{|\ms{X}_i|} = \deg_{\mc{A}}|_{|\ms{X}_i|} \cdot \operatorname{rank}_{\mc{O}_{\ms{X}_i}}(\mc{O}_{\ms{X}_i}^{\oplus r_i}) = n_i \cdot r_i = n.
    \end{align*}
    This shows that $\deg_{\mc{A}'}$ is identically equal to $n$ on $|\ms{X}|$, meaning that $\mc{A}'$ is an Azumaya algebra of constant rank $n^2$ over $\ms{X}$.
\end{proof}

\begin{lemma}\label{azumaya_algebra_torsion}\cite[1.2.9]{SHIN-THESIS2019}.
    If $\mc{A}$ is an Azumaya algebra of constant rank $n^2$ on $\ms{X}$, 
    then the class $[\ms{G}_{\mc{A}}] \in {\H}^2_{\sm}(\ms{X}, \G_m)$ is $n$-torsion.
\end{lemma}

\begin{lemma}\label{sm_top_et_top_coincide}
    For any sheaf $\ms{F}$ of groups on $\ms{X}$ in the big smooth topology,
    \begin{align*}
        \check \H^1_{\et}(\ms{X},\ms{F})\simeq \check \H^1_{\sm}(\ms{X},\ms{F}).
    \end{align*}
    And for any sheaf $\ms{F}$ of abelian groups on $\ms{X}$ in the big smooth topology and any $i\in\mathbb{N}$,
    \begin{align*}
        \H^i_{\et}(\ms{X},\ms{F})\simeq \H^i_{\sm}(\ms{X},\ms{F}).
    \end{align*}
\end{lemma}
\begin{proof}
    By the first statement in \cite[\href{https://stacks.math.columbia.edu/tag/021Y}{Tag 021Y}]{SP} (whose underlying reason is \cite[\href{https://stacks.math.columbia.edu/tag/055U}{Tag 055U}]{SP}),
    the big smooth topology defines the same topos as the big \'etale topology.
    Hence by \cite[III 3.6.5 (5)]{Giraud1971} and \cite[(2.9)]{LW23}
    \begin{align*}
        \check \H^1_{\sm}(\ms{X},\ms{F})&
        \simeq \mathrm{Tors}(\ms X_{\mathrm{sm}}, \ms F)_{/\cong}\simeq\mathrm{Tors}(\ms X_{\mathrm{\et}},\ms F)_{/\cong}
        =\check \H^1_{\et}(\ms{X},\ms{F}).
    \end{align*}
    And:
    \begin{align*}
        \H^i_{\et}(\ms{X},\ms{F})=R^i\Gamma_{\et}(-,\ms{F})\simeq R^i\Gamma_{\sm}(-,\ms{F})
        =\H^i_{\sm}(\ms{X},\ms{F}).
    \end{align*}
\end{proof}

\section{Purity and Torsionness of the Brauer group}

\begin{situation}\label{sit:setup}
$\mc{X}$ is a regular Noetherian algebraic stack with generically affine stabilizer.
\end{situation}

We prove that the Brauer group is torsion for $\mc{X}$ in \Cref{sit:setup}. 

\begin{remark}
We explain the condition ``generically affine stabilizer". Let $\mc X$ be a regular Noetherian algebraic stack. Then by \cite[Proposition 7.4.39]{alper} the residual gerbe $\mc G_x$ exists for every $x\in |\mc X|$ and $\mc G_x$ is a gerbe over the field $\kappa(x)$. 

We prove in \cref{prop:stack_integral} that each connected component of $\mc X$ is integral. 
Thus,
by \cite[\href{https://stacks.math.columbia.edu/tag/0GWC}{Tag 0GWC}]{SP},
we may denote the generic points of $\mc X$ by $\eta_i$. 
\iffalse
Since $\mc X$ is quasi-separated, 
we prove in \cref{prop:brauer-injective-generic} that $\mc G_{\eta_i}$ are quasi-separated gerbes over $\kappa(\eta_i)$.
\fi
In \cref{prop:vanishing-h1} we prove that $\mc G_{{\eta_i},L}=BG_i$ for some finite separable extension $L/\kappa(\eta_i)$ and some algebraic group $G_i$ over $L$. 
The condition ``generically affine stabilizer" means exactly the following: 
\begin{center}
    $G_i$ is a linear algebraic group for all $i$. 
\end{center}
By fppf descent of affineness, 
this is equivalent to saying that \textit{the stabilizers of generic points are linear algebraic groups}. 
\end{remark}

Torsionness of the Brauer group $\Br(\mc{X})$ under \Cref{sit:setup} can be reduced to the case where $\mc{X}$ is a gerbe over an integral, regular qcqs scheme $S$. By taking connected components we may first assume that $\mc{X}$ is integral. 
This is justified by \Cref{prop:stack_integral}. 

\begin{lemma}\label{lem:top-irreducible}
Let $T$ be a compact connected topological space. 
Suppose $T = \bigcup_{\alpha \in I} V_\alpha$ is covered by a collection of non-empty open irreducible subsets $V_\alpha \subseteq T$. 
Then $T$ is irreducible.
\end{lemma}

\begin{proof}
We may assume $I=\{1,\dots,m\}$ is finite. 
Notice that if $A$ and $B$ are non-empty open irreducible subsets with $A \cap B \neq \emptyset$, then $A \cap B$ is a non-empty open subset of $A$, so $A \cap B$ is irreducible and dense in $A$. Similarly, $A \cap B$ is dense in $B$, whence $A \cap B$ is dense in $A \cup B$. Since the closure of an irreducible set is irreducible, $A \cup B = \overline{A \cap B}^{A \cup B}$ is irreducible.

Since $T$ is connected, for every $i$, $V_i$ meets some $V_j$ where $j\neq i$. Since $V_i \cup V_j$ is irreducible, $T$ is covered by $m-1$ irreducible open subsets. 
$T$ is irreducible by induction on $m$.
\end{proof}

\begin{proposition}\label{prop:stack_integral}
Let $\mc{X}$ be a regular, connected, qcqs algebraic stack. Then $\mc{X}$ is an integral stack.
\end{proposition}

\begin{proof}
By \cite[\href{https://stacks.math.columbia.edu/tag/0GWD}{Tag 0GWD}]{SP} it suffices to show that $\mc X$ is reduced and irreducible. 
Since $\mc{X}$ is a regular quasi-compact algebraic stack, there exists a smooth presentation $\pi \colon U \to \mc{X}$ where $U$ is a quasi-compact regular scheme.
$U$ has finitely many connected components, say
\[
U = U_1 \sqcup U_2 \sqcup \dots \sqcup U_n,
\]
where each $U_i$ is an open and closed subscheme of $U$ and each $U_i$ is integral.

Since $U_i$ is irreducible and $|\pi|$ is a continuous open map, the image $V_i$ is an open irreducible subset. Thus, $|\mc{X}|$ is a connected topological space covered by open irreducible subsets. $|\mc{X}|$ is irreducible by \cref{lem:top-irreducible}. Finally, $U$ is regular, hence reduced. This implies that $\mc{X}$ is a reduced stack.
\end{proof}

\medskip

We start with a structure theorem for quasi-separated algebraic spaces. 

\begin{theorem}\label{thm:5.5.1}\cite[Theorem 5.5.1]{alper}.
Every quasi-separated algebraic space has a dense open subspace which is a scheme.
\end{theorem}

\begin{theorem}\label{thm:dense-open-gerbe}
Let $\mc{X}$ be an integral, qcqs, regular (automatically Noetherian) algebraic stack. Then there exists a dense open substack $\mc{U} \subset \mc{X}$ such that $\mc{U}$ is an fppf gerbe over a qcqs integral regular (Noetherian) scheme.
\end{theorem}

\begin{proof}
By generic flatness \cite[Exercise 4.3.32]{alper} applied to the inertia morphism, there exists a dense open substack $\mc{X}_1 \subset \mc{X}$ such that $I_{\mc{X}_1} \to \mc{X}_1$ is flat and of finite presentation, hence fppf. Applying \cite[Proposition 7.4.19]{alper} to $\mc{X}_1$, the fppf sheafification $X_1$ of isomorphism classes is an algebraic space, and the quotient morphism $\pi \colon \mc{X}_1 \to X_1$ is a gerbe. In particular, the map $\pi$ is smooth and surjective by \cite[Proposition 7.4.17]{alper}. Hence, by smooth descent, $X_1$ is an integral regular Noetherian algebraic space.

To show $X_1$ is quasi-separated,
it suffices to show its diagonal morphism $\Delta_{X_1} \colon X_1 \to X_1 \times X_1$ is quasi-compact. 
Consider the smooth surjective cover $\pi \times \pi \colon \mc{X}_1 \times \mc{X}_1 \to X_1 \times X_1$. Pulling back $\Delta_{X_1}$ along $\pi \times \pi$ yields the cartesian square:
\[
\begin{tikzcd}
\mc{X}_1 \times_{X_1} \mc{X}_1 \arrow[r, "j"] \arrow[d] & \mc{X}_1 \times \mc{X}_1 \arrow[d, "\pi \times \pi"] \\
X_1 \arrow[r, "\Delta_{X_1}"'] & X_1 \times X_1
\end{tikzcd}
\]
To show $\Delta_{X_1}$ is quasi-compact, it suffices by fppf descent to show that its base change $j \colon \mc{X}_1 \times_{X_1} \mc{X}_1 \to \mc{X}_1 \times \mc{X}_1$ is a quasi-compact morphism.

Note that $\Delta_{\mc{X}_1} \colon \mc{X}_1 \to \mc{X}_1 \times \mc{X}_1$ factors as
\[
\mc{X}_1 \xrightarrow{\quad \Delta_{\mc{X}_1 / X_1} \quad} \mc{X}_1 \times_{X_1} \mc{X}_1 \xrightarrow{\quad j \quad} \mc{X}_1 \times \mc{X}_1,
\]
where $\Delta_{\mc{X}_1 / X_1}$ is the relative diagonal. Because $\pi \colon \mc{X}_1 \to X_1$ is a gerbe, $\Delta_{\mc{X}_1 / X_1}$ is fppf by \cite[Exercise 7.4.18]{alper}, in particular $\Delta_{\mc{X}_1 / X_1}$ is surjective. Since $\mc{X}_1$ is quasi-separated, the composite $\Delta_{\mc{X}_1} = j \circ \Delta_{\mc{X}_1 / X_1}$ is a quasi-compact morphism, so the morphism $j$ is quasi-compact. 

By \cref{thm:5.5.1}, the quasi-separated algebraic space $X_1$ contains a dense open subspace $V \subset X_1$ which is a scheme.  The inverse image $\mc{U}:=\pi^{-1}(V)$ is a dense open substack of $\mc{X}_1$, hence a dense open substack of $\mc{X}$. Restricting the gerbe $\pi \colon \mc{X}_1 \to X_1$ to $V$, we conclude that $\mc{U} \to V$ is a gerbe, where $V$ is a quasi-separated integral regular Noetherian scheme.
\end{proof}

\begin{pg}\label{pg:reduction}
By purity results \cite[Theorem 2.1.7]{SHIN-THESIS2019}, the restriction map $\Br(\mc{X}) \to \Br(\mc{U})$ to a dense open substack $\mc{U} \subseteq \mc{X}$ is injective. Therefore, we may assume $\mc{X}$ is a gerbe over $V$. In the next proposition, we prove the corresponding purity statement $\Br(\mc{X}) \hookrightarrow \Br(\mc{X}_\eta)$, where $\mc{X}_\eta$ is the generic residual gerbe. 
\end{pg}

\begin{proposition}\label{prop:brauer-injective-generic}
Let $S$ be a regular Noetherian integral scheme with function field $K$ and generic point $\eta : \Spec(K) \rightarrow S$, and let $p \colon \mc{X} \to S$ be a quasi-separated gerbe over $S$. Let $\mc{X}_\eta = \mc{X} \times_S \eta$ be the pullback gerbe over $\eta$, and let $j \colon \mc{X}_\eta \to \mc{X}$ be the canonical inclusion morphism. Then there is a canonical injection 
$\Br(\mc{X}) \hookrightarrow \Br(\mc{X}_\eta)$.
\end{proposition}

\begin{proof}
The idea of this proof follows \cite{Grothendieck}. We show that there are canonical injections: 
\[
\Br(\mc{X}) = \H^2(\mc{X}, \mathbb{G}_{m, \mc{X}}) \xrightarrow{\quad\alpha\quad} \H^2(\mc{X}, j_* \mathbb{G}_{m, \mc{X}_\eta}) \xrightarrow{\quad\beta\quad} \H^2(\mc{X}_\eta, \mathbb{G}_{m, \mc{X}_\eta})
\]
Consider the divisor sequence of \'etale sheaves on $S$ given by 
\[
0 \to \mathbb{G}_{m, S} \to  \eta_*\mathbb{G}_{m, K} \to \underline{\mathrm{Div}}_S \to 0
\]
where $\underline{\mathrm{Div}}_S \simeq \bigoplus_{z \in S^{(1)}} (i_z)_* \underline{\mathbb{Z}}_z$ with $i_z \colon \Spec \kappa(z) \hookrightarrow S$. 
\[
\begin{tikzcd}[column sep=large, row sep=large]
\mathscr{X}_\eta \arrow[r, "j"] \arrow[d, "p_\eta"'] \arrow[dr, phantom, "\lrcorner"{pos=0.12}] & \mathscr{X} \arrow[d, "p"] \\
\operatorname{Spec}(K) \arrow[r, "\eta"'] & S
\end{tikzcd}
\qquad\qquad\qquad 
\begin{tikzcd}[column sep=large, row sep=large]
\mathscr{X}_z \arrow[r,  "i_z"] \arrow[d, "p_z"'] \arrow[dr, phantom, "\lrcorner"{pos=0.12}] & \mathscr{X} \arrow[d, "p"] \\
\operatorname{Spec}\kappa(z) \arrow[r,  "i_z"'] & S
\end{tikzcd}
\]
Pulling back the divisor sequence along $p$ yields the short exact sequence of sheaves on $\mc{X}$: 
\[
0 \to \mathbb{G}_{m, \mc{X}} \to \mc{R}^*_{\mc{X}} \to \underline{\mathrm{Div}}_{\mc{X}} \to 0
\]
where by \cite[\href{https://stacks.math.columbia.edu/tag/075H}{Tag 075H}]{SP} on the big \'etale site
$\mc{R}^*_{\mc{X}} := j_* \mathbb{G}_{m, \mc{X}_\eta} \simeq p^{-1}(\eta_* \mathbb{G}_{m, K})$, $\underline{\mathrm{Div}}_{\mc{X}} := p^{-1}\underline{\mathrm{Div}}_S \simeq \bigoplus_{z \in S^{(1)}} (i_z)_* \underline{\mathbb{Z}}_{\mc{X}_z}$ with $i_z \colon \mc{X}_z \hookrightarrow \mc{X}$ the base change inclusion over $z$. Taking cohomology on the big \'etale site gives the exact sequence 
\[
\H_{\et}^1(\mc{X}, \underline{\mathrm{Div}}_{\mc{X}}) \to \H_{\et}^2(\mc{X}, \mathbb{G}_{m, \mc{X}}) \xrightarrow{\alpha} \H_{\et}^2(\mc{X}, j_* \mathbb{G}_{m, \mc{X}_\eta}).
\]

Since $p \colon \mc{X} \to S$ is quasi-separated and $S$ is quasi-separated,
$\mc{X}$ is a quasi-separated algebraic stack; by \cite[\href{https://stacks.math.columbia.edu/tag/06R9}{Tag 06R9}]{SP} $\mc{X}$ is a quasi-compact algebraic stack.
Hence $\mc{X}_z$ is quasi-separated over $\kappa (z)$. 
So by \cite[\href{https://stacks.math.columbia.edu/tag/0GQV}{Tag 0GQV}]{SP},
$\H_{\et}^1(\mc{X}, \underline{\mathrm{Div}}_{\mc{X}}) \simeq \bigoplus_{z \in S^{(1)}} \H_{\et}^1(\mc{X}, (i_z)_*\underline{\mathbb{Z}}_{\mc{X}_z})$.
Applying the Leray spectral sequence \cite[\href{https://stacks.math.columbia.edu/tag/0732}{Tag 0732}]{SP} and its five-term sequence for $i_z: \ms X_z\to \ms X$ and $\underline{\mathbb{Z}}_{\ms X_z}$,
we obtain the injection 
$\H_{\et}^1(\mc{X}, (i_z)_*\underline{\mathbb{Z}}_{\mc{X}_z})\hookrightarrow \H_{\et}^1(\mc{X}_z, \underline{\mathbb{Z}}_{\mc{X}_z})$.
$\H_{\et}^1(\mc{X}_z, \underline{\mathbb{Z}}_{\mc{X}_z})=0$ by \cref{prop:vanishing-h1}.  
Thus, the first map $\alpha \colon \Br(\mc{X}) = \H_{\et}^2(\mc{X}, \mathbb{G}_{m, \mc{X}}) \hookrightarrow \H_{\et}^2(\mc{X}, j_* \mathbb{G}_{m, \mc{X}_\eta})$ is injective. 

\medskip

Next, applying the Leray seven-term sequence for $j \colon \mc{X}_\eta \to \mc{X}$ and $\mathbb{G}_{m, \mc{X}_\eta}$,
we have:
\[
\H_{\et}^0(\mc{X}, R^1 j_* \mathbb{G}_{m, \mc{X}_\eta}) \to \H_{\et}^2(\mc{X}, j_* \mathbb{G}_{m, \mc{X}_\eta}) \xrightarrow{\beta} \H_{\et}^2(\mc{X}_\eta, \mathbb{G}_{m, \mc{X}_\eta}).
\]
By base change along the gerbe projection $p \colon \mc{X} \to S$, 
again by \cite[\href{https://stacks.math.columbia.edu/tag/075H}{Tag 075H}]{SP} we have $R^1 j_* \mathbb{G}_{m, \mc{X}_\eta} \simeq p^{-1}(R^1 \eta_* \mathbb{G}_{m, K})$. 
But $R^1 \eta_* \mathbb{G}_{m, K} = 0$ by \cite[Lemma 1.6]{Grothendieck}, which implies that $R^1 j_* \mathbb{G}_{m, \mc{X}_\eta} = 0$ and hence the homomorphism $\beta \colon \H_{\et}^2(\mc{X}, j_* \mathbb{G}_{m, \mc{X}_\eta}) \hookrightarrow \H_{\et}^2(\mc{X}_\eta, \mathbb{G}_{m, \mc{X}_\eta})$ is injective. Combining $\alpha$ and $\beta$, the composite map $\Br(\mc{X}) \hookrightarrow \Br(\mc{X}_\eta)$ is injective.
\end{proof}

\begin{proposition}\label{prop:vanishing-h1}
Let $k$ be a field, and let $\mc{X} \to \Spec k$ be a quasi-separated gerbe over $(\Spec k)_{\fppf}$. Then
\[
\H^1_{\et}(\mc{X}, \underline{\mathbb{Z}}) = 0.
\]
\end{proposition}

First, we prove a special case where $\mc X=BG$. One crucial ingredient of the proof is the calculation of cohomological descent along a torsor, which is well studied in \cite[pp.~38-45]{Sansuc1981}. 

\iffalse
Let $G$ be a linear algebraic group over a field $k$ and $\pi \colon \ms X \to \ms Y$ be a $G$-torsor. Then we have the Čech spectral sequence relative to this fppf covering and to the sheaf $\G_m$:
$$
\check{\H}^p(\ms X/\ms Y, \ms{H}^q(\G_m)) \Rightarrow \H^{p+q}(\ms Y, \G_m),
$$
where $\ms{H}^q(\G_m)$ denotes the presheaf $V \mapsto H^q(V, \G_m)$ on $\ms Y_{\fppf}$, or equivalently $\ms Y_{\et}$. 
\fi

Let $G$ be a linear algebraic group over a field $k$ and $\pi \colon \ms X \to \ms Y$ be a $G$-torsor. 
\iffalse
Since $\ms X$ is a $G$-torsor over $\ms Y$, for each $i \ge 0$, we
\fi
We have an isomorphism: 
\[
\ms X \times_k G^i \xrightarrow{\sim} \ms X \times_{\ms Y} \dots \times_{\ms Y} \ms X = \ms X^{i+1}_{/\ms Y}
\]
\[(x, g_1, \dots, g_i) \mapsto (x, xg_1, \dots, xg_1 \dots g_i)
\]
The simplicial system $(\Sigma) = \{\ms X^{i+1}_{/\ms Y}\}$ is thus isomorphic to the simplicial system $\{\ms X \times_k G^i\}$ whose faces are defined by:
\begin{align*}
d_i^0(x, g_1, \dots, g_{i+1}) &= (x g_1, g_2, \dots, g_{i+1}), \\
d_i^j(x, g_1, \dots, g_{i+1}) &= (x, g_1, \dots, g_j g_{j+1}, \dots, g_{i+1}) \quad \text{for } 1 \le j \le i, \\
d_i^{i+1}(x, g_1, \dots, g_{i+1}) &= (x, g_1, \dots, g_i).
\end{align*}
If further $\ms X=\Spec k$ and the group action is trivial, we may omit the first entry: 
\begin{equation}\label{eq:differential}
\begin{split}
d_i^0(g_1, \dots, g_{i+1}) &= (g_2, \dots, g_{i+1}), \\
d_i^j(g_1, \dots, g_{i+1}) &= (g_1, \dots, g_j g_{j+1}, \dots, g_{i+1}) \quad \text{for } 1 \le j \le i, \\
d_i^{i+1}(g_1, \dots, g_{i+1}) &= (g_1, \dots, g_i).
\end{split}
\end{equation}

\begin{lemma}\label{lem:simplicial-iso}
Let $G$ be a discrete group over a field $k$. Assume $\ms X$ and $\ms Y$ are algebraic stacks over $k$ and $\ms X$ is a $G$-torsor over $\ms Y$. Then for any (2,1)-functor $F: (\mathrm{Stack}/k)^{\mathrm{op}} \rightarrow \mathrm{Ab}$, if $F(\coprod \ms X_i)=\prod F(\ms X_i)$, then there is a canonical isomorphism: 
\begin{equation*}\label{eq:cech-group-iso}
\check{\H}^p(\ms X^{\bullet}_{/\ms Y}, F) \simeq \H^p_{\mr{group}}(G, F(\ms X)),
\end{equation*}
where the $G$-action on $F(\ms X)$ is given by $F(\ms X) \xrightarrow{F(g)}F(\ms X)$. 
\end{lemma}

\begin{proof}
The isomorphism $\ms X \times_{k} G^i \xrightarrow{\sim} \ms X^{i+1}_{/\ms Y}$ is given by $(x, g_1, \dots, g_i) \mapsto (x, x g_1, \dots, x g_1 \dots g_i)$. 
Applying $F$ to $\ms X \times_{k} G^i = \coprod_{G^i} \ms X$ gives $F(\ms X \times_{k} G^i) = \prod_{G^i} F(\ms X) = \Hom_{\text {Set}}(G^i, F(\ms X))$. By \cref{eq:differential}, the \v{C}ech complex coincides with the nonhomogeneous cochain complex for group cohomology with coefficient $F(\ms X)$.
\end{proof}

\begin{lemma}\label{lemma:h1-bg-z}
Let $k$ be a field, and let $G$ be an algebraic group over $k$ whose connected components are geometrically connected. Let $\mathcal{Y} = BG = [\Spec k / G]$ be the classifying stack of $G$ over $k$, and let $f \colon \ms X = \Spec k \to \mathcal{Y}$ be the canonical $G$-torsor. Then $\H^1_{\et}(BG, \underline{\mathbb{Z}}) = 0.$
\end{lemma}

\begin{proof}
Since $\ms X\to \ms{Y}$ is an fppf covering,
we have the Čech spectral sequence relative to this fppf covering and to the sheaf $\underline{\mathbb{Z}}$ by \cite[\href{https://stacks.math.columbia.edu/tag/03OW}{Tag 03OW}]{SP}:
\[
\check{\H}^p(\ms X^{\bullet}_{/\ms Y}, \ms{H}_{\fppf}^q(\underline{\mathbb{Z}})) \Rightarrow \H_{\fppf}^{p+q}(\ms Y, \underline{\mathbb{Z}}),
\]
where $\ms{H}_{\fppf}^q(\underline{\mathbb{Z}})$ denotes the presheaf $V \mapsto \H_{\fppf}^q(V, \underline{\mathbb{Z}})$ on $\ms Y_{\fppf}$.
By \cite[\href{https://stacks.math.columbia.edu/tag/0DGH}{Tag 0DGH}]{SP},
$\ms{H}_{\fppf}^q(\underline{\mathbb{Z}})=\ms{H}_{\et}^q(\underline{\mathbb{Z}}): V\mapsto \H_{\fppf}^q(V, \underline{\mathbb{Z}})=\H_{\et}^q(V, \underline{\mathbb{Z}})$.
Choose a smooth surjective presentation $U\to \ms Y$.
Applying \cite[\href{https://stacks.math.columbia.edu/tag/06XJ}{Tag 06XJ}, \href{https://stacks.math.columbia.edu/tag/0DGH}{Tag 0DGH}]{SP}, 
we obtain $\H_{\fppf}^{p+q}(\ms Y, \underline{\mathbb{Z}})=\H_{\et}^{p+q}(\ms Y, \underline{\mathbb{Z}})$.
So we have the following \v{C}ech spectral sequence:
\begin{equation}\label{cech_spectral_et}
    E_2^{p,q}=\check{\H}^p(\ms X^{\bullet}_{/\ms Y}, \ms{H}_{\et}^q(\underline{\mathbb{Z}})) \Rightarrow \H_{\et}^{p+q}(\ms Y, \underline{\mathbb{Z}}).
\end{equation}

\iffalse
By the comparison between fppf and \'etale cohomology for the constant sheaf $\underline{\mathbb Z}$,
we may identify its terms and abutment with the corresponding étale cohomology groups.
\fi

Let $\ms X_{\bullet}$ be the \v{C}ech nerve of $\ms X_{/ \mathcal{Y}}$. For each $p \ge 0$, we have an isomorphism of $k$-schemes:
\[
\ms X_p = \underbrace{\ms X \times_{\mathcal{Y}} \ms X \times_{\mathcal{Y}} \dots \times_{\mathcal{Y}} \ms X}_{p+1 \text{ times}} \simeq G^p \times_k \ms X \simeq G^p.
\]
\iffalse
Consider the spectral sequence:
\[
E_1^{p,q} = \H^q_{\et}(\ms X_p, \underline{\mathbb{Z}}) \implies \H^{p+q}_{\et}(\mathcal{Y}, \underline{\mathbb{Z}}),
\]
whose $E_2$-page is given by
\[
E_2^{p,q} = \check{\H}^p(\ms X/\mathcal{Y}, \mathcal{H}^q_{\et}(\underline{\mathbb{Z}})) \implies \H^{p+q}_{\et}(\mathcal{Y}, \underline{\mathbb{Z}}).
\]
\fi
From the above spectral sequence, we obtain the low-degree exact sequence:
\begin{equation*}
0 \to E_2^{1,0} \to \H^1_{\et}(\mathcal{Y}, \underline{\mathbb{Z}}) \to E_2^{0,1} \to E_2^{2,0}.
\end{equation*}
To establish the vanishing of $\H_{\et}^1(\mc{Y}, \underline{\mathbb{Z}})$, it suffices to show that $E_2^{1,0} = 0$ and $E_2^{0,1} = 0$.

\medskip

For $E_2^{0,1}$: 
    \[
    E_2^{0,1} = \check{\H}^0(\ms X_{\bullet}, \mathcal{H}^1_{\et}(\underline{\mathbb{Z}})) = \ker\left( \H^1_{\et}(\ms X, \underline{\mathbb{Z}}) \to \H^1_{\et}(\ms X_1, \underline{\mathbb{Z}}) \right).
    \]
    $\H^1_{\et}(\ms X, \underline{\mathbb{Z}}) = \H^1_{\text{group}}(\text{Gal}(k^{\sep}/k),\mathbb Z)=\Hom_{\text{group}}(\text{Gal}(k^{\sep}/k),\mathbb Z)=0$.
    Thus $E_2^{0,1} = 0.$

\medskip

For $E_2^{1,0}$: 
    \[
    E_2^{1,0} = \check{\H}^1(\ms X_{\bullet}, \mathcal{H}^0_{\et}(\underline{\mathbb{Z}})).
    \]
We compute this term using the explicit Čech complex $C^\bullet$ with $C^p = \H_{\et}^0(\ms X_p, \underline{\mathbb{Z}})$:
\[
0 \to C^0 \xrightarrow{d_0} C^1 \xrightarrow{d_1} C^2 \to \dots
\]
Since $\ms X_p \simeq G^p$, $\H_{\et}^0(\ms X_p, \underline{\mathbb{Z}})=\H_{\et}^0(G^p, \underline{\mathbb{Z}})$ consists of locally constant functions on $G^p$. Let $\pi_0$ be the functor of connected components of the underlying space. Since connected components of $G$ are geometrically connected, we have $\pi_0(G^p) \simeq \pi_0(G)^p$ as abstract groups. 
\begin{align*}
C^0 &= \H_{\et}^0(\Spec k, \underline{\mathbb{Z}}) = \mathbb{Z}, \\
C^1 &= \H_{\et}^0(G, \underline{\mathbb{Z}}) \simeq \operatorname{Map}(\pi_0(G), \mathbb{Z}), \\
C^2 &= \H_{\et}^0(G \times G, \underline{\mathbb{Z}}) \simeq \operatorname{Map}(\pi_0(G) \times \pi_0(G), \mathbb{Z}).
\end{align*}
We calculate the differential maps based on \Cref{eq:differential}: 
\begin{enumerate}
    \item[\textnormal{(i)}] \textit{The differential $d_0 \colon C^0 \to C^1$}: The face maps $d_0^0, d_0^1 \colon \ms X_1 \simeq G \to \ms X_0 \simeq \Spec k$ are both the structure morphism $G \to \Spec k$. For any constant $c \in C^0 = \mathbb{Z}$, its image is
    \[
    d_0(c) = (d_0^0)^*(c) - (d_0^1)^*(c) = c - c = 0.
    \]
    Therefore, $\im(d_0) = 0$.

    \item[\textnormal{(ii)}] \textit{The differential $d_1 \colon C^1 \to C^2$}: The face maps $d_1^0, d_1^1, d_1^2 \colon G \times G \to G$ are given by $d_1^0(g_1, g_2) = g_2$, $d_1^1(g_1, g_2) = g_1 g_2$ (group multiplication), and $d_1^2(g_1, g_2) = g_1$. For any function $f \in C^1 \simeq \operatorname{Map}(\pi_0(G), \mathbb{Z})$, we have
    \[
    d_1(f)(g_1, g_2) = (d_1^0)^*(f) - (d_1^1)^*(f) + (d_1^2)^*(f) = f(g_2) - f(g_1 g_2) + f(g_1).
    \]
    Thus, $f \in \ker(d_1)$ if and only if $f(g_1 g_2) = f(g_1) + f(g_2)$ for all $g_1, g_2 \in \pi_0(G)$, which means $f$ is an abstract group homomorphism from $\pi_0(G)$ to $(\mathbb{Z}, +)$.
\end{enumerate}
Thus, 
$
E_2^{1,0} = \frac{\ker(d_1)}{\im(d_0)} \simeq \operatorname{Hom}_{\text{group}}(\pi_0(G), \mathbb{Z}).
$
Finally, since $\pi_0(G)$ is finite, it is a torsion group. Therefore $E_2^{1,0} = \operatorname{Hom}_{\text{group}}(\pi_0(G), \mathbb{Z})= 0$, which proves the proposition. 
\end{proof}

\medskip

\begin{proof}[Proof of \cref{prop:vanishing-h1}]

Since the gerbe morphism is smooth  by \cite[\href{https://stacks.math.columbia.edu/tag/0DN8}{Tag 0DN8}]{SP}, 
there exists an \'etale section of $\mc X \rightarrow \Spec k$, 
i.e. there exists an \'etale scheme over $k$ of the form $\widetilde{L} =\bigsqcup \Spec L_i$ with $L_i/k$ being finite separable extensions such that $\mc X\times_k \widetilde{L}\simeq BG$ for some fppf sheaf $G$ over $\widetilde{L}$.
Choose one $L_i$, say $L$, which is a finite separable extension of $k$. 
Then $\ms{X}_L$ is neutral over $L$. 
Since \(\mathcal X\) is quasi-separated over \(k\), 
the following cartesian diagram shows that \(G\) is a qcqs group algebraic space of finite type over $L$:
\[
\begin{tikzcd}
    G \arrow[r] \arrow[d] & \Spec L \arrow[d]\\
    \ms{X}_L \arrow[r,"\Delta"] & \ms{X}_L\times_L \ms{X}_L
\end{tikzcd}
\]
Hence \(G\) is an algebraic group by \cite[Theorem 5.5.28]{alper}.
After replacing $L$ by a further finite separable extension, we may assume that the connected components of $G$ are geometrically connected 
by \cite[\href{https://stacks.math.columbia.edu/tag/054R}{Tag 054R}, \href{https://stacks.math.columbia.edu/tag/054S}{Tag 054S}]{SP}.

\iffalse
Choose a smooth surjective atlas \(U \to \mathcal X\) with \(U\) a nonempty \(k\)-scheme.
Since \(U\) is smooth over \(k\) (because the gerbe morphism $\ms X\to k$ is smooth \cite[\href{https://stacks.math.columbia.edu/tag/0DN8}{Tag 0DN8}]{SP}), there exists a closed point \(u \in U\) whose residue field
\(L:=\kappa(u)\) is a finite separable extension of \(k\) by \cite[\href{https://stacks.math.columbia.edu/tag/056U}{Tag 056U}]{SP}.
The composite morphism
\[
\Spec L \longrightarrow U \longrightarrow \mathcal X
\]
gives an object of \(\mathcal X(L)\), hence the gerbe \(\mathcal X_L\) is neutral.
Therefore \(\mathcal X_L \simeq BG\) for some fppf sheaf of groups \(G\) over \(L\).
Since \(\mathcal X\) is quasi-separated over \(k\), 
by the following cartesian diagram:
\[
\begin{tikzcd}
    G \arrow[r] \arrow[d] & \Spec L \arrow[d]\\
    \ms{X}_L \arrow[r,"\Delta"] & \ms{X}_L\times_L \ms{X}_L
\end{tikzcd}
\]
\(G\) is a qcqs group algebraic space of finite type over $L$ by \cite[\href{https://stacks.math.columbia.edu/tag/04YW}{Tag 04YW}]{SP}, hence \(G\) is an algebraic group
by \cite[Theorem 5.5.28]{alper}.
After replacing \(L\) by a further finite separable extension, we may assume that the connected components of \(G\) are geometrically connected by \cite[\href{https://stacks.math.columbia.edu/tag/054R}{Tag 054R}, \href{https://stacks.math.columbia.edu/tag/054S}{Tag 054S}]{SP}.
\fi

Enlarging $L$ further, we may assume that $L/k$ is finite Galois.
Let $\Gamma := \operatorname{Gal}(L/k)$ denote the Galois group of $L/k$. Let $\mc{X}_{L} := \mc{X} \times_{k} \Spec L$. The projection $\mc{X}_{L} \to \mc{X}$ is a $\Gamma$-Galois cover. By the spectral sequence \cref{cech_spectral_et} and \cref{lem:simplicial-iso}, 
\[
E_2^{p,q} = \H^p_{\text{group}}\big(\Gamma, \, \H_{\et}^q(\mc{X}_{L}, \underline{\mathbb{Z}})\big) \implies \H_{\et}^{p+q}(\mc{X}, \underline{\mathbb{Z}}).
\]

\iffalse
where $\H_{\et}^q(\mc{X}_{L}, \underline{\mathbb{Z}})$ is cohomology of trivial $\Gamma-$module. 
\fi

The low-degree exact sequence gives:
\[
0 \longrightarrow E_2^{1,0} \longrightarrow \H_{\et}^1(\mc{X}, \underline{\mathbb{Z}}) \longrightarrow E_2^{0,1}.
\]
It suffices to show that $E_2^{1,0} = 0$ and $E_2^{0,1} = 0$.

\medskip

For $E_2^{1,0}$: 
since $\mc{X}_{L}=BG$, $\H_{\et}^0(\mc{X}_{L}, \underline{\mathbb{Z}}) = \mathbb{Z}$.
Since the Galois group is finite,
\[
E_2^{1,0} = \H_{\text{group}}^1(\Gamma, \mathbb{Z}) = \Hom_{\mathrm{group}}(\Gamma, \mathbb{Z})=0.
\]

\medskip

For $E_2^{0,1}$, 
applying \cref{lemma:h1-bg-z} to $\mc X_L$ over $L$, we have $
\H_{\et}^1(\mc{X}_{L}, \underline{\mathbb{Z}}) = 0.
$ 
Therefore $
E_2^{0,1} = \H^0\big(\Gamma, \, \H_{\et}^1(\mc{X}_{L}, \underline{\mathbb{Z}})\big)=0
$.

\end{proof}

\begin{theorem}\label{thm:br_torsion}
Let $\mc{X}$ be an algebraic stack satisfying \cref{sit:setup}. Then $\Br(\mc{X}) = \H^2_{\et}(\mc{X}, \mathbb{G}_m)$ is a torsion group.
\end{theorem}

\begin{proof}
By \cref{pg:reduction}, we reduce to the case where $\mc{X}$ is a gerbe over a regular integral qcqs scheme $S$. Since $\mc{X}$ is quasi-separated, $f \colon \mc{X} \to S$ is a quasi-separated morphism.

Let $K = K(S)$ be the function field of $S$, and let $\eta = \Spec(K)$ be the generic point of $S$. Let $\mc{X}_\eta = \mc{X} \times_S \eta$ be the pullback gerbe over $\eta$. By \cref{prop:brauer-injective-generic}, we have an injection $\Br(\mc{X}) \hookrightarrow \Br(\mc{X}_\eta)$. Therefore, it suffices to show that $\Br(\mc{X}_\eta) = \H^2_{\et}(\mc{X}_\eta, \mathbb{G}_m)$ is a torsion group.

\medskip

There exists a finite Galois extension $L/K$ such that $\mc{X}_\eta \times_K \Spec L=:\mc X_L=BG$ for some algebraic group $G$ over $L$. Since $f \colon \mc{X} \to S$ is quasi-separated with generically affine stabilizer, $G$ is affine. Let $\Gamma = \mathrm{Gal}(L/K)$. Consider the spectral sequence associated with the $\Gamma$-Galois cover $\mc{X}_{L} \to \mc{X}_\eta$ (\cref{lem:simplicial-iso}):
\[
E_2^{p,q} = \H^p\left(\Gamma, \H^q_{\et}(\mc{X}_{L}, \mathbb{G}_m)\right) \implies \H^{p+q}_{\et}(\mc{X}_\eta, \mathbb{G}_m).
\]
We analyze the terms on the $E_2$-page that contribute to $\H^2_{\et}(\mc{X}_\eta, \mathbb{G}_m)$:
\begin{enumerate}
    \item \emph{For $p \ge 1$}: Since $\Gamma$ is a finite group, the Galois cohomology group $\H^p(\Gamma, M)$ is torsion for any $p \ge 1$ and any $\Gamma$-module $M$. Thus $E_2^{1,1}$ and $E_2^{2,0}$ are both torsion groups.
    \item \emph{For $p = 0$}: $E_2^{0,2} = \H^0\left(\Gamma, \H^2_{\et}(\mc{X}_{L}, \mathbb{G}_m)\right) = \H^2_{\et}(\mc{X}_{L}, 
    \mathbb{G}_m)^\Gamma$. 
    \iffalse
    The gerbe $\mc{X}_{L} \simeq BG$ for a linear group $G$, then by \cite[Lemma 3.9]{LW23}, $\H^2_{\et}(\mc X_L, \mathbb{G}_m)$ is a torsion group, hence $E_2^{0,2}$ is also torsion.
    \fi
    It suffices to show that $\H^2_{\et}(\mc{X}_{L}, 
    \mathbb{G}_m)$ is torsion.
    Choose a closed immersion of linear $L$-groups $G\hookrightarrow \GL_n$ and consider $Y:=\GL_n/G$.
    Then $\ms{X}_L=BG=[Y/\GL_n]$ and $Y$ is a smooth finite type
    $L$-variety (see \cite[Theorem 3.7]{conrad}).
    Moreover, 
    $Y$ is geometrically integral: $Y_{\overline{L}}$ is irreducible because $\GL_{n,\overline{L}}\to Y_{\overline{L}}$ is surjective and $\GL_{n,\overline{L}}$ is irreducible.
    Note that
    the proof of \cite[Theorem 3.1]{LW23} applies to $Y\to \ms{X}_L=[Y/\GL_n]$ over an arbitrary field $L$.
Hence we have the following exact sequence:
\begin{align*}
    \mathrm{Pic}(\GL_n)\to \H^2_{\et}(\mc{X}_{L}, 
    \mathbb{G}_m)\to \H^2_{\et}(Y, 
    \mathbb{G}_m)
\end{align*}
By \cite[Corollary 1.8]{Grothendieck}, $\H^2_{\et}(Y, \mathbb{G}_m)$ is torsion.
By $\mathrm{Pic}(\GL_n)=0$,
$\H^2_{\et}(\mc{X}_{L}, \mathbb{G}_m)$ is torsion.

\end{enumerate}
It follows that all $E_2^{p,q}$ terms contributing to $\H^2_{\et}(\mc{X}_\eta, \mathbb{G}_m)$ are torsion, which implies that the limit $\H^2_{\et}(\mc{X}_\eta, \mathbb{G}_m) = \Br(\mc{X}_\eta)$ is a torsion group. This completes the proof.
\end{proof}

Combining \cref{thm:dense-open-gerbe}, \cref{pg:reduction} and \cref{prop:brauer-injective-generic} we have the following purity result: 

\begin{theorem}{Generic Purity}
Let $\mc{X}$ be an integral, regular Noetherian algebraic stack. Let $\mc G$ be the generic gerbe. Then we have an injection $\Br(\mc X) \hookrightarrow \Br(\mc G)$. 
\end{theorem}

\begin{remark}
    The condition “generically affine stabilizer” is crucial. By \cite[Proposition 2.4.4]{SHIN-THESIS2019}, if $E$ is an elliptic curve over $\mathbb C$, then we have an isomorphism
\[
\H_{\et}^2(BE, \mathbb{G}_m) \simeq \Pic^0(E)
\]
where $\Pic^0(E)$ is non-torsion. 
\end{remark}

\medskip

We construct a non-quasi-separated quasi-compact regular  counterexample to \cref{thm:br_torsion}. 

\begin{pg}\label{pg:construction}
Let $X = \A^2_{\mathbb Q} = \Spec \mathbb Q[x, y]$. Consider the discrete group $G = \mathbb Z^2 = \mathbb Z \oplus \mathbb Z$. Define the action of $G$ on $X$ by translations:
$$
(n, m) \cdot (x, y) = (x + n, y + m), \quad \text{for all } (n, m) \in \mathbb Z^2 \text{ and } (x, y) \in \A^2_{\mathbb Q}.
$$
Since $G$ acts freely on $X$, the quotient sheaf in the \'etale topology defines an algebraic space:
$$
Y := \A^2_{\mathbb Q} / \mathbb Z^2.
$$
\end{pg}

\begin{proposition}\label{prop:properties}
The algebraic space $Y$ constructed in \cref{pg:construction} satisfies the following properties:
\begin{enumerate}
    \item $Y$ is regular of finite type over $\mathbb Q$. 
    \item The stabilizer of $Y$ at every point is trivial, hence affine.
    \item $Y$ is not quasi-separated over $\mathbb{Q}$.
\end{enumerate}
\end{proposition}

\begin{proof}
(1) and (2) are trivial. 

For (3), consider the diagonal morphism $\Delta_Y: Y \to Y \times_{\mathbb Q} Y$. Its pullback along the \'etale cover $X \times_{\mathbb Q} X \to Y \times_{\mathbb Q} Y$ yields the equivalence relation scheme:
    \[
    R := X \times_Y X \simeq \mathbb Z^2 \times \A^2_{\mathbb Q} = \coprod_{(n, m) \in \mathbb Z^2} \A^2_{\mathbb Q}.
    \]
    $R$ is not quasi-compact while $X \times_{\mathbb Q} X$ is quasi-compact, so $\Delta_Y$ is not quasi-compact, proving that $Y$ is not quasi-separated.
\end{proof}

Now we compute the terms $E_2^{p,q}$ in the \v{C}ech spectral sequence for $p + q = 2$:

\begin{proposition}\label{prop:e2-terms}
For the cover $\pi: X=\A^2_{\mathbb Q} \to Y$, consider the spectral sequence $E^{p,q}_2=\check{\H}^p(X^{\bullet}_{/Y}, \mathcal{H}^q_{\et}(\G_m)) \implies \H^{p+q}_{\et}(Y, \G_m)$. Then:
\begin{enumerate}
    \item $E_2^{2,0} \simeq \mathbb Q^\times$,
    \item $E_2^{0,1} = E_2^{1,1} = 0$.
\end{enumerate}
\end{proposition}

\begin{proof}
We calculate each term using \cref{lem:simplicial-iso}, noting that 
the functor \(F=\H^q_{\et}(-,\mathbb G_m)\) satisfies the condition $F(\coprod X_i)=\prod F(X_i)$. 

(1) We have $\H^0_{\et}(\A^2_{\mathbb Q}, \G_m) = \mc{O}(\A^2_{\mathbb Q})^\times = \mathbb Q^\times$. The action of $G$ on constant functions is trivial. 
By the universal coefficient formula for group cohomology \cite[Exercise 6.1.5(3)]{weibel1994}, for any group $G$ and any trivial $G$-module $M$, there exists a split short exact sequence:
$$
0 \longrightarrow \text{Ext}^1_{\mathbb{Z}}\bigl(\mathrm{H}_{n-1}(G, \mathbb{Z}), M\bigr) \longrightarrow \mathrm{H}^n(G, M) \longrightarrow \Hom_{\mathbb{Z}}\bigl(\mathrm{H}_n(G, \mathbb{Z}), M\bigr) \longrightarrow 0.
$$

Set $n = 2$, $G = \mathbb{Z}^2$, and $M = \mathbb{Q}^\times$. By \cite[Theorem 6.1.11]{weibel1994}, the first homology group is $
    \mathrm{H}_1(\mathbb{Z}^2, \mathbb{Z}) \simeq (\mathbb{Z}^2)^{\mathrm{ab}} \simeq \mathbb{Z}^2.
$
Since $\mathrm{H}_1(\mathbb{Z}^2, \mathbb{Z}) \simeq \mathbb{Z}^2$ is a free $\mathbb{Z}$-module (hence projective), its $\text {Ext}^1$ functor vanishes:
    \[
    \text{Ext}^1_{\mathbb{Z}}\bigl(\mathrm{H}_1(\mathbb{Z}^2, \mathbb{Z}), \mathbb{Q}^\times\bigr) \simeq \text{Ext}^1_{\mathbb{Z}}(\mathbb{Z}^2, \mathbb{Q}^\times) = 0.
    \]
Thus, 
$$
\mathrm{H}^2(\mathbb{Z}^2, \mathbb{Q}^\times) \xrightarrow{\sim} \Hom_{\mathbb{Z}}\bigl(\mathrm{H}_2(\mathbb{Z}^2, \mathbb{Z}), \mathbb{Q}^\times\bigr).
$$
Write $\mathbb{Z}^2 = \mathbb{Z} \times \mathbb{Z}$ and apply the Künneth formula \cite[Proposition 6.1.13]{weibel1994}:
\begin{align*}
0 \to \bigoplus_{p+q=2} \mathrm{H}_p(\mathbb{Z}, \mathbb{Z}) \otimes_{\mathbb{Z}} \mathrm{H}_q(\mathbb{Z}, \mathbb{Z}) \to \mathrm{H}_2(\mathbb{Z}^2, \mathbb{Z}) \to \bigoplus_{p+q=1} \text{Tor}_1^{\mathbb{Z}}\bigl(\mathrm{H}_p(\mathbb{Z}, \mathbb{Z}), \mathrm{H}_q(\mathbb{Z}, \mathbb{Z})\bigr) \to 0.
\end{align*}

Let $G = \langle t \rangle \cong \mathbb{Z}$ be the infinite cyclic group generated by $t$. The group ring of $G$ over $\mathbb{Z}$ is $\mathbb{Z}[G] \cong \mathbb{Z}[t, t^{-1}]$. 
Consider the short exact sequence of $\mathbb{Z}[G]$-modules:
\begin{equation*} \label{eq:free_resolution}
0 \longrightarrow \mathbb{Z}[G] \xrightarrow{\quad d_1 \quad} \mathbb{Z}[G] \xrightarrow{\quad \epsilon \quad} \mathbb{Z} \longrightarrow 0,
\end{equation*}
where $\epsilon \colon \mathbb{Z}[G] \to \mathbb{Z}$ is the augmentation map defined by $\epsilon(t^k) = 1$ for all $k \in \mathbb{Z}$, and $d_1 \colon \mathbb{Z}[G] \to \mathbb{Z}[G]$ is defined by $d_1(x) = (t - 1)x$. 
By direct calculation:
\[
\mathrm{H}_0(\mathbb{Z}, \mathbb{Z}) \simeq \mathbb{Z}, \quad \mathrm{H}_1(\mathbb{Z}, \mathbb{Z}) \simeq \mathbb{Z}, \quad \mathrm{H}_2(\mathbb{Z}, \mathbb{Z}) \simeq 0
\]
Thus
$
\mathrm{H}_2(\mathbb{Z}^2, \mathbb{Z}) \simeq \mathrm{H}_1(\mathbb{Z}, \mathbb{Z}) \otimes_{\mathbb{Z}} \mathrm{H}_1(\mathbb{Z}, \mathbb{Z}) \simeq \mathbb{Z} \otimes_{\mathbb{Z}} \mathbb{Z} \simeq \mathbb{Z}
$.
Hence $E_2^{2,0}=\H^2(\mathbb Z^2, \H^0_{\et}(\A^2_{\mathbb Q}, \G_m))=\mathrm{H}^2(\mathbb{Z}^2, \mathbb{Q}^\times)=\Hom_{\mathbb{Z}}\bigl(\mathbb Z, \mathbb{Q}^\times\bigr)=\mathbb{Q}^\times$. 

\medskip

(2) By \cite[Theorem 3.6]{Traverso1970} $\H^1_{\et}(\A^2_{\mathbb Q}, \G_m) = \Pic(\A^2_{\mathbb Q}) = 0$; thus 
$E_2^{p,1} = \H^p(\mathbb Z^2, 0) = 0$. 
\end{proof}

\begin{theorem}
There is an injection $\mathbb Q^\times \hookrightarrow \Br(Y)$.
In particular, $\Br(Y)$ contains $\mathbb Q^\times$ as a subgroup, and therefore contains non-torsion elements.
\end{theorem}

\begin{proof}
    By the five-term exact sequence:
$$
0 \longrightarrow E_2^{1,0} \longrightarrow \H^1 \longrightarrow E_2^{0,1} \longrightarrow E_2^{2,0} \longrightarrow \H^2.
$$
Since $E_2^{0,1}=0$, the map $E_2^{2,0}\to \H^2$ is injective.
\end{proof}
\begin{remark}
This calculation confirms that the quasi-separatedness condition in \cref{sit:setup} cannot be omitted.
\end{remark}

\end{appendix}
\bibliographystyle{alphaurl}
\bibliography{reference.bib}

\end{document}